\documentclass[12pt,reqno,twoside]{amsart}
\usepackage{amssymb,amsmath,amsthm,soul,color,paralist}
\usepackage{t1enc}
\usepackage[cp1250]{inputenc}
\usepackage{a4,indentfirst,latexsym}
\usepackage{graphics}
\usepackage{mathrsfs}
\usepackage{cite,enumitem,graphicx}
\usepackage[colorlinks=true,urlcolor=blue,
citecolor=red,linkcolor=blue,linktocpage,pdfpagelabels,
bookmarksnumbered,bookmarksopen]{hyperref}
\usepackage[english]{babel}
\usepackage[left=2.50cm,right=2.50cm,top=2.72cm,bottom=2.72cm]{geometry}
\usepackage[metapost]{mfpic}
\usepackage[colorinlistoftodos]{todonotes}
\usepackage[normalem]{ulem}

\makeatletter
\providecommand\@dotsep{5}
\def\listtodoname{List of Todos}
\def\listoftodos{\@starttoc{tdo}\listtodoname}
\makeatother

\numberwithin{equation}{section}

\newcommand{\e}{\varepsilon}
\newcommand{\R}{\mathbb{R}}
\newcommand{\N}{\mathbb{N}}
\newcommand{\C}{\mathbb{C}}

\DeclareMathOperator{\dive}{div}
\DeclareMathOperator{\supp}{supp}
\DeclareMathOperator{\B}{\mathcal{B}}

\newcommand{\h}{H^{s}_{\e}}
\newcommand{\2}{2^{*}_{s}}

\newtheorem{theorem}{Theorem}[section]
\newtheorem{lemma}[theorem]{Lemma}

\theoremstyle{definition}

\theoremstyle{remark}
\newtheorem{remark}[theorem]{Remark}

\numberwithin{equation}{section}

\begin{document}
\title[Fractional Choquard equations with magnetic fields]{Concentration phenomena for a fractional Choquard equation with magnetic field}

\author[V. Ambrosio]{Vincenzo Ambrosio}
\address{Vincenzo Ambrosio\hfill\break\indent 
Department of Mathematics  \hfill\break\indent
EPFL SB CAMA \hfill\break\indent
Station 8 CH-1015 Lausanne, Switzerland}
\email{vincenzo.ambrosio2@unina.it}


\date{}

%\dedicatory{Communicated by Y. Charles Li, received June 1, 2018.}


\begin{abstract}
We consider the following  nonlinear fractional Choquard equation
$$
\e^{2s}(-\Delta)^{s}_{A/\e} u + V(x)u = \e^{\mu-N}\left(\frac{1}{|x|^{\mu}}*F(|u|^{2})\right)f(|u|^{2})u \mbox{ in } \R^{N},
$$
where $\varepsilon>0$ is a parameter, $s\in (0, 1)$, $0<\mu<2s$, $N\geq 3$, $(-\Delta)^{s}_{A}$ is the fractional magnetic Laplacian, $A:\R^{N}\rightarrow \R^{N}$ is a smooth magnetic potential, $V:\R^{N}\rightarrow \R$ is a positive potential with a local minimum and $f$ is a continuous nonlinearity with subcritical growth. 
By using variational methods we prove the existence and concentration of nontrivial solutions for $\e>0$ small enough.
\end{abstract}

\maketitle
%\tableofcontents

\section{Introduction}

\noindent
In this paper we investigate the existence and concentration of nontrivial solutions for the following nonlinear fractional Choquard equation
\begin{equation}\label{P}
\e^{2s}(-\Delta)^{s}_{A/\e} u + V(x)u = \e^{\mu-N}\left(\frac{1}{|x|^{\mu}}*F(|u|^{2})\right)f(|u|^{2})u \,\mbox{ in } \,\R^{N},
\end{equation}
where $\e>0$ is a parameter, $s\in (0,1)$, $N\geq 3$, $0<\mu<2s$, $V:\R^{N}\rightarrow \R$ is a continuous potential and 
$A:\R^{N}\rightarrow \R^{N}$ is a $C^{0, \alpha}$ magnetic potential, with $\alpha\in (0,1]$.

The nonlocal operator $(-\Delta)^{s}_{A}$ is the fractional magnetic Laplacian which may be defined for any $u:\R^{N}\rightarrow \C$ smooth enough by setting
$$
(-\Delta)^{s}_{A}u(x)= c_{N,s} P.V. \int_{\R^{N}} \frac{u(x)-u(y)e^{\imath A(\frac{x+y}{2})\cdot (x-y)}}{|x-y|^{N+2s}} \,dy  \quad (x\in \R^{N}),
$$
where  $c_{N, s}$ is a normalizing constant. This operator has been introduced in \cite{DS, I} with motivations falling into the framework of the general theory of L\'evy processes.
As showed in \cite{SV}, when $s\rightarrow 1$, the operator $(-\Delta)^{s}_{A}$ reduces to the magnetic Laplacian (see \cite{LaL, LL}) defined as
$$
\left(\frac{1}{\imath}\nabla-A\right)^{2}\!u= -\Delta u -\frac{2}{\imath} A(x) \cdot \nabla u + |A(x)|^{2} u -\frac{1}{\imath} u \dive(A(x)),
$$ 
which has been widely investigated by many authors: see \cite{AF, AFF, AS, Cingolani, CS, EL, K}.

Recently, many papers dealt with different fractional problems involving the operator $(-\Delta)^{s}_{A}$. 
d'Avenia and Squassina \cite{DS} studied the existence of ground states solutions for some fractional magnetic problems via minimization arguments.
Pinamonti et al. \cite{PSV1, PSV2} obtained a magnetic counterpart of the Bourgain-Brezis-Mironescu formula and the MazŐya-Shaposhnikova formula respectively; see also \cite{NPSV} for related results.
Zhang et al. \cite{ZSZ} proved a multiplicity result for a fractional magnetic Schr\"odinger equation with critical growth. 
In \cite{MPSZ} Mingqi et al. studied existence and multiplicity of solutions for a subcritical fractional Schr\"{o}dinger-Kirchhoff equation involving an external magnetic potential. 
Fiscella et al. \cite{FPV}  considered a fractional magnetic problem in a bounded domain proving the existence of at  least two nontrivial weak solutions under suitable assumptions on the nonlinear term. 
In \cite{AD} the author and d'Avenia used variational methods and Ljusternick-Schnirelmann theory to prove existence and multiplicity of  nontrivial solutions for a fractional Schr\"{o}dinger equation with subcritical nonlinearities. \\
We note that when $A=0$, the operator $(-\Delta)^{s}_{A}$ becomes the celebrated fractional Laplacian $(-\Delta)^{s}$ which arises in the study of several physical phenomena like phase transitions, crystal dislocations, quasi-geostrophic flows, flame propagations  and so on. Due to the extensive literature on this topic, we refer the interested reader to \cite{DPMV, DPV, MBRS} and the references therein.\\
In absence of the magnetic field, equation \eqref{P} is a fractional  Choquard equation of the type
\begin{equation}\label{FChE}
(-\Delta)^{s} u + V(x)u = \left(\frac{1}{|x|^{\mu}}*F(u)\right)f(u) \,\mbox{ in } \,\R^{N}.
\end{equation}
d'Avenia et al. \cite{DSS} studied the existence, regularity and asymptotic behavior of solutions to \eqref{FChE} when $f(u)=u^{p}$ and $V(x)\equiv const$. If $V(x)=1$ and $f$ satisfies Berestycki-Lions type assumptions, the existence of ground state solutions for a fractional Choquard equation has been established in \cite{SGY}.
The analyticity and radial symmetry of positive ground state for a critical boson star equation has been considered by Frank and Lenzmann in \cite{FL}.  
Recently, the author in \cite{Apota} studied the multiplicity and concentration of positive solutions for a fractional Choquard equation under local conditions on the potential $V(x)$. \\
When $s=1$, equation \eqref{FChE} reduces to  the generalized  Choquard equation:
\begin{equation}\label{GCE}
-\Delta u + V(x) u = \left(\frac{1}{|x|^{\mu}}*F(u)\right)f(u) \,\mbox{ in } \,\R^{N}.
\end{equation}
If $p=\mu=2$, $V(x)\equiv 1$, $F(u)=\frac{u^{2}}{2}$ and $N=3$, \eqref{GCE} is called the Choquard-Pekar equation  which goes back  to the 1954's work by Pekar \cite{Pek} to the description of a polaron at rest in Quantum Field Theory and to 1976's model of Choquard of an electron trapped in its own hole as an approximation to Hartree-Fock theory for a one-component plasma \cite{LS}. The same equation was proposed by Penrose \cite{Pen} as a model of self-gravitating matter and is known in that context as the Schr\"odinger-Newton equation.
Lieb in \cite{Lieb} proved the existence and uniqueness of positive solutions to a Choquard-Pekar equation. 
Subsequently, Lions \cite{Lions} established a multiplicity result via variational methods. Ackermann in \cite{Ack} proved the existence and multiplicity of solutions for \eqref{GCE} when $V$ is periodic.
Ma and Zhao \cite{MZ} showed that, up to translations, positive solutions of equation \eqref{GCE} with $f(u)=u^{p}$, are radially symmetric and monotone decreasing for suitable values of $\mu$, $N$ and $p$. This results has been improved by Moroz and Van Schaftingen in \cite{MVS1}. 
The same authors in \cite{MVS2} obtained the existence of ground state solutions with a general nonlinearity $f$. 
Cingolani et al. \cite{CSS} showed the existence of multi-bump type solutions for a Schro\"odinger equation in presence of electric and magnetic potentials and Hartree-type nonlinearities.
Alves et al. \cite{AFY}, inspired by \cite{AFF, CSS}, studied the multiplicity and concentration phenomena of solutions for \eqref{GCE} in presence of a magnetic field. For a more detailed bibliography on the Choquard equation we refer to \cite{MVS3}.

Motivated by \cite{AFY, Apota, AD}, in this paper we focus our attention on the existence and concentration of solutions to \eqref{P} under local conditions on the potential $V$. Before stating our main result, we introduce the assumptions on $V$ and $f$.
Along the paper, we assume that the potential $V: \R^{N}\rightarrow \R$ is a continuous function verifying the following conditions introduced in \cite{DPF}:
\begin{enumerate}
\item [$(V_1)$] $V(x)\geq V_{0}>0$ for all $x\in \R^{N}$;
\item [$(V_2)$] there exists a bounded open set $\Lambda\subset \R^{N}$ such that 
$$
V_{0}=\inf_{x\in \Lambda} V(x)<\min_{x\in \partial \Lambda} V(x), 
$$
\end{enumerate}
and $f: \R\rightarrow \R$ is a continuous function such that $f(t)=0$ for $t<0$ and satisfies the following assumptions: 
\begin{enumerate}
\item [($f_1$)] $\displaystyle{\lim_{t\rightarrow 0} f(t)=0}$;
\item [($f_2$)] there exists $q\in (2, \frac{2^{*}_{s}}{2}(2-\frac{\mu}{N}))$, where $2^{*}_{s}=\frac{2N}{N-2s}$, such that $\displaystyle{\lim_{t\rightarrow \infty} \frac{f(t)}{t^{\frac{q-2}{2}}}=0}$;
\item [($f_3$)] the map $\displaystyle{t \mapsto f(t)}$ is increasing for every $t>0$.
\end{enumerate}
We point to that the restriction on $q$ in $(f_2)$ is related to the Hardy-Littlewood-Sobolev inequality:
\begin{theorem}\label{HLS}\cite{LL}
Let $r, t>1$ and $0<\mu<N$ such that $\frac{1}{r}+\frac{\mu}{N}+\frac{1}{t}=2$. Let $f\in L^{r}(\R^{N})$ and $h\in L^{t}(\R^{N})$. Then there exists a sharp constant $C(r, N, \mu, t)>0$ independent of $f$ and $h$ such that 
$$
\int_{\R^{N}}\int_{\R^{N}} \frac{f(x)h(y)}{|x-y|^{\mu}}\, dx dy\leq C(r, N, \mu, t)\|f\|_{L^{r}(\R^{N})}\|h\|_{L^{t}(\R^{N})}.
$$
\end{theorem}
Indeed, by $(f_1)$ and $(f_2)$ it follows that $|F(|u|^{2})|\leq C(|u|^{2}+|u|^{q})$, so it is easy to check that 
the term
\begin{equation}\label{chiara}
\left|\int_{\R^{N}} \left( \frac{1}{|x|^{\mu}}* F(|u|^{2})\right) F(|u|^{2}) dx\right|<\infty \quad \forall u\in \h,
\end{equation}
where $\h$ is defined in Section $2$, when $F(|u|^{2})\in L^{t}(\R^{N})$ for all $t>1$ such that
$$
\frac{2}{t}+\frac{\mu}{N}=2, \, \mbox{ that is } t=\frac{2N}{2N-\mu}.
$$ 
Therefore, if $q\in (2, \frac{2^{*}_{s}}{2}(2-\frac{\mu}{N}))$ and $\mu\in (0, 2s)$ we can use the fractional Sobolev embedding $H^{s}(\R^{N}, \R)\subset L^{r}(\R^{N}, \R)$ for all $r\in [2, \2]$, to deduce that $tq\in (2, 2^{*}_{s})$ and then \eqref{chiara} holds true. \\
Now, we can state the main result of this paper:
\begin{theorem}\label{thm1}
Suppose that $V$ verifies $(V_1)$-$(V_2)$,  $0<\mu<2s$ and $f$ satisfies $(f_1)$-$(f_3)$ with $q\in (2, 2\frac{(N-\mu)}{N-2s} )$. Then there exists $\e_{0}>0$ such that, for any $\e\in (0, \e_{0})$, problem  \eqref{P} has a nontrivial solution. Moreover, if $|u_{\e}|$ denotes one of these solutions and $x_{\e}\in \R^{N}$ its global maximum, then 
$$
\lim_{\e\rightarrow 0} V(x_{\e})=V_{0},
$$
and
\begin{align*}
|u_{\e}(x)|\leq \frac{\tilde{C} \e^{N+2s}}{\e^{N+2s}+|x-x_{\e}|^{N+2s}} \quad \forall x\in \R^{N}.
\end{align*}
\end{theorem}
\begin{remark}
Assuming $f\in C^{1}$, one can use Ljusternick-Schnirelmann theory and argue as in \cite{Apota, AD} to relate the number of nontrivial solutions to \eqref{P} with the topology of the set where the potential attains its minimum value.
\end{remark}
The proof of Theorem \ref{thm1} is inspired by some variational arguments used in \cite{AFF, AFY, AM, Apota}. Anyway, the presence of the fractional magnetic Laplacian and nonlocal Hartree-type nonlinearity does not permit to easily adapt  in our setting the techniques developed in the above cited papers and, as explained in what follows, a more intriguing and accurate analysis will be needed. Firstly, after a change of variable, it is easy to check that problem (\ref{P}) is equivalent to the following one:
\begin{equation}\label{R}
(-\Delta)^{s}_{A_{\e}} u + V_{\e}(x)u =  \left(\frac{1}{|x|^{\mu}}*F(|u|^{2})\right)f(|u|^{2})u  \mbox{ in } \R^{N} 
\end{equation}
where $A_{\e}(x):=A(\e x)$ and $V_{\e}(x):=V(\e x)$. 
In the spirit of \cite{DPF} (see also \cite{AFF, AM}), we modify the nonlinearity in a suitable way and we consider an auxiliary problem. We note that the restriction imposed on $\mu$ allows us to use the penalization technique.
Without loss of generality, along the paper we will assume that $0\in \Lambda$ and $V_{0}=V(0)=\inf_{x\in\R^{N}} V(x)$.
Now, we fix $\ell>0$ large enough, which will be determined later on, and let $a>0$ be the unique number such that $f(a)=\frac{V_{0}}{\ell}$. Moreover, we introduce the functions
$$
\tilde{f}(t):=
\begin{cases}
f(t)& \text{ if $t \leq a$} \\
\frac{V_{0}}{\ell}   & \text{ if $t >a$},
\end{cases}
$$ 
and
$$
g(x, t):=\chi_{\Lambda}(x)f(t)+(1-\chi_{\Lambda}(x))\tilde{f}(t),
$$
where $\chi_{\Lambda}$ is the characteristic function on $\Lambda$, and  we write $G(x, t)=\int_{0}^{t} g(x, \tau)\, d\tau$.\\
From assumptions $(f_1)$-$(f_3)$, it is easy to verify that $g$ fulfills the following properties:
\begin{enumerate}
\item [($g_1$)] $\displaystyle{\lim_{t\rightarrow 0} g(x, t)=0}$ uniformly in $x\in \R^{N}$;
\item [($g_2$)] $\displaystyle{\lim_{t\rightarrow \infty} \frac{g(x, t)}{t^{\frac{q-2}{2}}}=0}$ uniformly in $x\in \R^{N}$;
\item [($g_3$)] $(i)$ $0\leq  G(x, t)< g(x, t)t$ for any $x\in \Lambda$ and $t>0$, and \\
$(ii)$ $0<G(x, t)\leq g(x, t)t\leq \frac{V_{0}}{\ell}t$ for any $x\in \R^{N}\setminus \Lambda$ and $t>0$,
\item [($g_4$)] $t\mapsto g(x, t)$ and $t\mapsto \frac{G(x, t)}{t}$ are increasing for all $x\in \R^{N}$ and $t>0$.
\end{enumerate}
Thus, we consider the following auxiliary problem 
\begin{equation*}
(-\Delta)^{s}_{A_{\e}} u + V_{\e}(x)u =  \left(\frac{1}{|x|^{\mu}}*G(\e x, |u|^{2})\right)g(\e x, |u|^{2})u \mbox{ in } \R^{N},
\end{equation*}
and in view of the definition of $g$, we are led to seek solutions $u$ of the above problem such that 
\begin{equation}\label{ue}
|u(x)|<a \mbox{ for all } x\in \R^{N}\setminus \Lambda_{\e}, \quad \mbox{ where }  \Lambda_{\e}:=\{x\in \R^{N}: \e x\in \Lambda\}.
\end{equation}
By using this penalization technique and establishing some careful estimates on the convolution term, we are able to prove that the energy functional associated with the auxiliary problem has a mountain pass geometry and satisfies the Palais-Smale condition; see Lemma \ref{MPG}, \ref{lemK} and \ref{PSc}. Then we can apply the Mountain Pass Theorem \cite{AR} to obtain the existence of a nontrivial solution $u_{\e}$ to the modified problem. 
The H\"older regularity assumption on the magnetic field $A$ and the fractional diamagnetic inequality \cite{DS}, will be properly exploited to show an interesting and useful relation between the mountain pass minimax level $c_{\e}$ of the modified functional and the minimax level $c_{V_{0}}$ associated with the limit functional; see Lemma \ref{AMlem1}. 
In order to verify that $u_{\e}$ is also solution of the original problem \eqref{P}, we need to check that $u_{\e}$ verifies \eqref{ue} for $\e>0$ sufficiently small. 
To achieve our goal, we first use an appropriate Moser iterative scheme \cite{Moser} to show that $\|u_{\e}\|_{L^{\infty}(\R^{N})}$ is bounded uniformly  with respect to $\e$. In these estimates, we take care of the fact that the convolution term 
is a bounded term in view of Lemma \ref{lemK}.
After that, we use these informations to develop a very clever approximation argument related in some sense to the following fractional version of Kato's inequality \cite{Kato}
$$
(-\Delta)^{s}|u|\leq \Re(sign(u)(-\Delta)^{s}_{A}u),
$$
to show that $|u_{\e}|$ is a weak subsolution to the problem
$$
(-\Delta)^{s}|u|+V(x)|u|= h(|u|^{2})|u| \mbox{ in } \R^{N},
$$
for some subcritical nonlinearity $h$, and then we prove that $|u_{\e}(x)|\rightarrow 0$ as $|x|\rightarrow \infty$, uniformly in $\e$; see Lemma \ref{moser}.  We point out that our arguments are different from the ones used in the classical case $s=1$ and the fractional setting $s\in (0,1)$ without magnetic field. Indeed, we don't know if a Kato's inequality is available in our framework, so we can not proceed as in \cite{CS, K} in which the Kato's inequality is combined with some standard elliptic estimates to obtain informations on the decay of solutions. Moreover, the appearance of magnetic field $A$  and the nonlocal character of $(-\Delta)^{s}_{A}$ do not permit to adapt the iteration argument developed in \cite{AFF, AFY} where $s=1$ and $A\not \equiv 0$, and we can not use the well-known estimates based on the Bessel kernel  (see \cite{AM, FQT}) established for fractional Schr\"odinger equations with $A=0$.
However, we believe that the ideas contained here can be also applied to deal with other fractional magnetic problems like \eqref{P}. Finally, we also give an estimate on the decay of modulus of solutions to \eqref{P} which is in clear accordance with the results in \cite{FQT}.
To the best of our knowledge, this is the first time that the penalization method is used to study nontrivial solutions for fractional Choquard equations with magnetic fields, and this represents the novelty of this work.\\
The paper is organized as follows: in Section $2$ we present some preliminary results and we collect some useful lemmas.
The Section $3$ is devoted to the proof of Theorem \ref{thm1}.

\section{Preliminaries and functional setting}
\noindent
For any $s\in (0,1)$, we denote by $\mathcal{D}^{s, 2}(\R^{N}, \R)$ the completion of $C^{\infty}_{0}(\R^{N}, \R)$ with respect to
$$
[u]^{2}=\iint_{\R^{2N}} \frac{|u(x)-u(y)|^{2}}{|x-y|^{N+2s}} \, dx \, dy =\|(-\Delta)^{\frac{s}{2}} u\|^{2}_{L^{2}(\R^{N})},
$$
that is
$$
\mathcal{D}^{s, 2}(\R^{N}, \R)=\left\{u\in L^{2^{*}_{s}}(\R^{N},\R): [u]_{H^{s}(\R^{N})}<\infty\right\}.
$$
Let us introduce the fractional Sobolev space
$$
H^{s}(\R^{N}, \R)= \left\{u\in L^{2}(\R^{N}) : \frac{|u(x)-u(y)|}{|x-y|^{\frac{N+2s}{2}}} \in L^{2}(\R^{2N}) \right \}
$$
endowed with the natural norm 
$$
\|u\| = \sqrt{[u]^{2} + \|u\|_{L^{2}(\R^{N})}^{2}}.
$$
Let us denote by $L^{2}(\R^{N}, \C)$ the space of complex-valued functions with summable square, endowed with the real
scalar product 
$$
\langle u, v\rangle_{L^{2}}=\Re\left(\int_{\R^{N}} u \bar{v} dx\right)
$$
for all $u, v\in L^{2}(\R^{N}, \C)$.
We consider the space
$$
\mathcal{D}^{s}_{A}(\R^{N}, \C)=\{u\in L^{2^{*}_{s}}(\R^{N}, \C) : [u]_{A}<\infty\}
$$
where
$$
[u]_{A}^{2}=\iint_{\R^{2N}} \frac{|u(x)-u(y)e^{\imath A(\frac{x+y}{2})\cdot (x-y)}|^{2}}{|x-y|^{N+2s}} dx dy.
$$
Then, we define the following fractional magnetic Sobolev space
$$
H^{s}_{A}(\R^{N}, \C)=\{u\in L^{2}(\R^{N}, \C): [u]_{A}<\infty\}.
$$
It is easy to check that $H^{s}_{A}(\R^{N}, \C)$ is a Hilbert space with the real scalar product
\begin{align*}
\langle u, v\rangle_{s, A}&=\Re\iint_{\R^{2N}} \frac{(u(x)-u(y)e^{\imath A(\frac{x+y}{2})\cdot (x-y)})\overline{(v(x)-v(y)e^{\imath A(\frac{x+y}{2})\cdot (x-y)})}}{|x-y|^{N+2s}} dx dy\\
&\quad +\langle u, v\rangle_{L^{2}}
\end{align*}
for any $u, v\in H^{s}_{A}(\R^{N}, \C)$. Moreover, $C^{\infty}_{c}(\R^{N}, \C)$ is dense in $H^{s}_{A}(\R^{N}, \C)$ (see \cite{AD}).
Now, we recall the following useful results:
\begin{theorem}\label{Sembedding}\cite{DS}
The space $H^{s}_{A}(\R^{N}, \C)$ is continuously embedded into $L^{r}(\R^{N}, \C)$ for any $r\in [2, \2]$ and compactly embedded into $L^{r}(K, \C)$ for any $r\in [1, \2)$ and any compact $K\subset \R^{N}$.
\end{theorem}

\begin{lemma}\label{DI}\cite{DS}
For any $u\in H^{s}_{A}(\R^{N}, \C)$, we get $|u|\in H^{s}(\R^{N},\R)$ and it holds
$$
[|u|]\leq [u]_{A}.
$$
We also have the following pointwise diamagnetic inequality 
$$
||u(x)|-|u(y)||\leq |u(x)-u(y)e^{\imath A(\frac{x+y}{2})\cdot (x-y)}| \mbox{ a.e. } x, y\in \R^{N}.
$$
\end{lemma}

\begin{lemma}\label{aux}\cite{AD}
If $u\in H^{s}(\R^{N}, \R)$ and $u$ has compact support, then $w=e^{\imath A(0)\cdot x} u \in H^{s}_{A}(\R^{N}, \C)$.
\end{lemma}

\noindent
For any $\e>0$, we denote by 
$$
\h=\left\{u\in \mathcal{D}^{s}_{A_{\e}}(\R^{N}, \C): \int_{\R^{N}} V_{\e}(x) |u|^{2}\, dx<\infty\right\}
$$
endowed with the norm 
$$
\|u\|^{2}_{\e}=[u]^{2}_{A_{\e}}+\|\sqrt{V_{\e}} |u|\|^{2}_{L^{2}(\R^{N})}.
$$

\noindent
From now on, we consider the following auxiliary problem 
\begin{equation}\label{Pe}
(-\Delta)^{s}_{A_{\e}} u + V_{\e}(x)u =  \left(\frac{1}{|x|^{\mu}}*G(\e x, |u|^{2})\right)g(\e x, |u|^{2})u \mbox{ in } \R^{N}
\end{equation}
and we note that if $u$ is a solution of (\ref{Pe}) such that 
\begin{equation*}
|u(x)|<a \mbox{ for all } x\in \R^{N}\setminus \Lambda_{\e},
\end{equation*}
then $u$ is indeed solution of the original problem (\ref{R}).

It is clear that weak solutions to (\ref{Pe}) can be found as critical points of the Euler-Lagrange functional $J_{\e}: H^{s}_{\e}\rightarrow \R$ defined by
$$
J_{\e}(u)=\frac{1}{2}\|u\|^{2}_{\e}-\frac{1}{4}\int_{\R^{N}} \left(\frac{1}{|x|^{\mu}}*G(\e x, |u|^{2})\right)G(\e x, |u|^{2})\, dx.
$$
We begin proving that $J_{\e}$ possesses a mountain pass geometry \cite{AR}. 
\begin{lemma}\label{MPG}
$J_{\e}$ has a mountain pass geometry, that is
\begin{enumerate}
\item [$(i)$] there exist $\alpha, \rho>0$ such that $J_{\e}(u)\geq \alpha$ for any $u\in H^{s}_{\e}$ such that $\|u\|_{\e}=\rho$;
\item [$(ii)$] there exists $e\in H^{s}_{\e}$ with $\|e\|_{\e}>\rho$ such that $J_{\e}(e)<0$.
\end{enumerate}
\end{lemma}
\begin{proof}
By using $(g_1)$ and $(g_2)$ we know that for any $\eta>0$ there exists $C_{\eta}>0$ such that
\begin{equation}\label{g-estimate}
|g(\e x, t)|\leq \eta +C_{\eta} |t|^{\frac{q-2}{2}}.
\end{equation}
In view of Theorem \ref{HLS} and \eqref{g-estimate}, we can deduce that
\begin{align}\label{a1}
\left|\int_{\R^{N}} \left(\frac{1}{|x|^{\mu}}*G(\e x, |u|^{2})\right)G(\e x, |u|^{2})\, dx\right| \leq C\left(\int_{\R^{N}} (|u|^{2}+|u|^{q}\, dx)^{t}\right)^{\frac{2}{t}},
\end{align}
where $\frac{1}{t}=\frac{1}{2}(2-\frac{\mu}{N})$. Since $2<q<\frac{2^{*}_{s}}{2}(2-\frac{\mu}{N})$ we have $t q\in (2, 2^{*}_{s})$ and by using Theorem \ref{Sembedding}  we can see that
\begin{align}\label{a2}
\left(\int_{\R^{N}} (|u|^{2}+|u|^{q}\, dx)^{t}\right)^{\frac{2}{t}}\leq C(\|u\|^{2}_{\e}+\|u\|^{q}_{\e})^{2}.
\end{align}
Putting together \eqref{a1} and \eqref{a2} we get
\begin{align*}
\left|\int_{\R^{N}} \left(\frac{1}{|x|^{\mu}}*G(\e x, |u|^{2})\right)G(\e x, |u|^{2})\, dx\right|\leq C(\|u\|^{2}_{\e}+\|u\|^{q}_{\e})^{2}\leq C(\|u\|^{4}_{\e}+\|u\|^{2q}_{\e}).
\end{align*}
Hence
$$
J(u)\geq \frac{1}{2}\|u\|^{2}_{\e}-C(\|u\|^{4}_{\e}+\|u\|^{2q}_{\e}),
$$
and recalling that $q>2$ we can infer that $(i)$ is satisfied.\\
Now, take a nonnegative function $u_{0}\in H^{s}(\R^{N}, \R)\setminus\{0\}$ with compact support such that $supp(u_{0})\subset \Lambda_{\e}$. Then, by Lemma \ref{aux} we know that $u_{0}(x)e^{\imath A(0)\cdot x}\in H^{s}_{\e}\setminus\{0\}$. Set
$$
h(t)=\mathfrak{F}\left(\frac{t u_{0}}{\|u_{0}\|_{\e}}\right) \mbox{ for } t>0,
$$
where 
$$
\mathfrak{F}(u)=\frac{1}{4}\int_{\R^{N}} \left( \frac{1}{|x|^{\mu}}*F(|u|^{2}) \right) F(|u|^{2})\,dx.
$$
From $(f_3)$ we know that $F(t)\leq f(t)t$ for all $t>0$.
Then, being $G(\e x, |u_{0}|^{2})=F(|u_{0}|^{2})$, we deduce that
\begin{align}\label{a3}
\frac{h'(t)}{h(t)}\geq \frac{4}{t} \quad \forall t>0.
\end{align}
Integrating \eqref{a3} over $[1, t\|u_{0}\|_{\e}]$ with $t>\frac{1}{\|u_{0}\|_{\e}}$, we get
$$
\mathfrak{F}(t u_{0})\geq \mathfrak{F}\left(\frac{u_{0}}{\|u_{0}\|_{\e}}\right) \|u_{0}\|_{\e}^{4}t^{4}.
$$
Summing up
$$
J_{\e}(t u_{0})\leq C_{1} t^{2}-C_{2}t^{4} \mbox{ for } t>\frac{1}{\|u_{0}\|_{\e}}.
$$
Taking $e=t u_{0}$ with $t$ sufficiently large, we can see that $(ii)$ holds.
\end{proof}

\noindent
Denoting by $c_{\e}$ the mountain pass level of the functional $J_{\e}$ and recalling that $supp(u_{0})\subset \Lambda_{\e}$, we can find $\kappa>0$ independent of $\e, l, a$ such that 
$$
c_{\e}= \inf_{u\in H^{s}_{\e}\setminus \{0\}}\max_{t\geq 0} J_{\e}(tu) <\kappa
$$
for all $\e>0$ small.
Now, let us define
$$
\B=\{u\in H^{s}(\R^{N}): \|u\|^{2}_{\e}\leq 4(\kappa+1)\}
$$
and we set
$$
\tilde{K}_{\e}(u)(x)=\frac{1}{|x|^{\mu}}*G(\e x, |u|^{2}).
$$
The next lemma is very useful because allows us to treat the convolution term as a bounded term.
\begin{lemma}\label{lemK}
Assume that $(f_1)$-$(f_3)$ hold and $2<q<\frac{2(N-\mu)}{N-2s}$. Then there exists $\ell_{0}>0$ such that 
$$
\frac{\sup_{u\in \B} \|\tilde{K}_{\e}(u)(x)\|_{L^{\infty}(\R^{N})}}{\ell_{0}}<\frac{1}{2} \mbox{ for any } \e>0.
$$
\end{lemma}
\begin{proof}
Let us prove that there exists $C_{0}>0$ such that 
\begin{equation}\label{a6}
\sup_{u\in \B} \|\tilde{K}_{\e}(u)(x)\|_{L^{\infty}(\R^{N})}\leq C_{0}.
\end{equation}
First of all, we can observe that
\begin{equation}\label{a5}
|G(\e x, |u|^{2})|\leq |F(|u|^{2})|\leq C(|u|^{2}+|u|^{q}) \mbox{ for all } \e>0.
\end{equation}
Hence, by using \eqref{a5}, we can see that
\begin{align}\label{a7}
|\tilde{K}_{\e}(u)(x)|
&\leq \Bigl| \int_{|x-y|\leq 1} \frac{F(|u|^{2})}{|x-y|^{\mu}} \,dy\Bigr|+\Bigl| \int_{|x-y|>1} \frac{F(|u|^{2})}{|x-y|^{\mu}} \,dy\Bigr| \nonumber\\
&\leq C \int_{|x-y|\leq 1} \frac{|u(y)|^{2}+|u(y)|^{q}}{|x-y|^{\mu}}\, dy+C \int_{\R^{N}} (|u|^{2}+|u|^{q})\, dy \nonumber\\
&\leq C \int_{|x-y|\leq 1} \frac{|u(y)|^{2}+|u(y)|^{q}}{|x-y|^{\mu}}\, dy+C
\end{align}
where in the last line we used Theorem \ref{Sembedding} and $\|u\|^{2}_{\e}\leq 4(\kappa+1)$.\\
Now, we take 
$$
t\in \Bigl(\frac{N}{N-\mu}, \frac{N}{N-2s}\Bigr] \mbox{ and } r\in \Bigl(\frac{N}{N-\mu}, \frac{2N}{q(N-2s)}\Bigr].
$$
By applying H\"older inequality, Theorem \ref{Sembedding} and $\|u\|^{2}_{\e}\leq 4(\kappa+1)$ we get 
\begin{align}\label{a8}
\int_{|x-y|\leq 1} \frac{|u(y)|^{2}}{|x-y|^{\mu}}\, dy&\leq \Bigl(\int_{|x-y|\leq 1} |u|^{2t}\, dy  \Bigr)^{\frac{1}{t}} \Bigl(\int_{|x-y|\leq 1} \frac{1}{|x-y|^{\frac{t\mu}{t-1}}}\, dy \Bigr)^{\frac{t-1}{t}}\nonumber \\
&\leq C_{*}(4(\kappa+1))^{2} \Bigl(\int_{\rho\leq 1} \rho^{N-1-\frac{t \mu}{t-1}}\, d\rho  \Bigr)^{\frac{t-1}{t}}<\infty,
\end{align}
because of $N-1-\frac{t \mu}{t-1}>-1$.
In similar fashion we can prove 
\begin{align}\label{a9}
\int_{|x-y|\leq 1} \frac{|u(y)|^{q}}{|x-y|^{\mu}}\, dy&\leq \Bigl(\int_{|x-y|\leq 1} |u|^{rq}\, dy  \Bigr)^{\frac{1}{r}} \Bigl(\int_{|x-y|\leq 1} \frac{1}{|x-y|^{\frac{r\mu}{r-1}}}\, dy \Bigr)^{\frac{r-1}{r}}\nonumber \\
&\leq C_{*}(4(\kappa+1))^{q} \Bigl(\int_{\rho\leq 1} \rho^{N-1-\frac{r \mu}{r-1}}\, d\rho  \Bigr)^{\frac{r-1}{r}}<\infty
\end{align}
in view of $N-1-\frac{r \mu}{r-1}>-1$.
Putting together \eqref{a8} and \eqref{a9} we obtain
$$
\int_{|x-y|\leq 1} \frac{|u(y)|^{2}+|u(y)|^{q}}{|x-y|^{\mu}}\, dy\leq C \mbox{ for all } x\in \R^{N}
$$
which together with \eqref{a7} implies \eqref{a6}.
Then we can find $\ell_{0}>0$ such that
$$
\frac{\sup_{u\in \B} \|\tilde{K}_{\e}(u)(x)\|_{L^{\infty}(\R^{N})}}{\ell_{0}}\leq \frac{C_{0}}{\ell_{0}}< \frac{1}{2}.
$$
\end{proof}

\noindent
Let $\ell_{0}$ be as in Lemma \ref{lemK} and $a>0$ be the unique number such that 
$$
f(a)=\frac{V_{0}}{\ell_{0}}.
$$
From now on we consider the penalized problem \eqref{Pe} with these choices.

\noindent
In what follows, we show that $J_{\e}$ verifies a local compactness condition.
\begin{lemma}\label{PSc}
 $J_{\e}$ satisfies the $(PS)_{c}$ condition for all $c\in [c_{\e}, \kappa]$.
\end{lemma}
\begin{proof}
Let $(u_{n})$ be a Palais-Smale sequence at the level $c$, that is $J_{\e}(u_{n})\rightarrow c$ and $J_{\e}'(u_{n})\rightarrow 0$. Let us note that $(u_{n})$ is bounded and there exists $n_{0}\in \N$ such that $\|u_{n}\|^{2}_{\e}\leq 4(\kappa+1)$ for all $n\geq n_{0}$. Indeed, by using $(g_3)$ and Lemma \ref{lemK}, we can see that
\begin{align*}
c+o_{n}(1)\|u_{n}\|_{\e}\geq J_{\e}(u_{n})-\frac{1}{4}\langle J'_{\e}(u_{n}), u_{n}\rangle\geq \frac{1}{4} \|u_{n}\|^{2}_{\e}
\end{align*}
which implies the thesis.\\
Now, we divide the proof in two main steps.\\
{\bf Step $1$}: For any $\eta>0$ there exists $R=R_{\eta}>0$ such that 
\begin{equation}\label{DF}
\limsup_{n\rightarrow \infty}\int_{\R^{N}\setminus B_{R}} \int_{\R^{N}} \frac{|u_{n}(x)-u_{n}(y)e^{\imath A_{\e}(\frac{x+y}{2})\cdot (x-y)}|^{2}}{|x-y|^{N+2s}}dx dy+\int_{\R^{N}\setminus B_{R}}V(\e x)|u_{n}|^{2}\, dx<\eta.
\end{equation}
Since $(u_{n})$ is bounded in $\h$, we may assume that $u_{n}\rightharpoonup u$ in $\h$ and $|u_{n}|\rightarrow |u|$ in $L^{r}_{loc}(\R^{N})$ for any $r\in [2, 2^{*}_{s})$. 
Moreover, by Lemma \ref{lemK}, we can deduce that
\begin{equation}\label{Kbound}
\frac{\sup_{n\geq n_{0}}\|\tilde{K}_{\e}(u_{n})(x)\|_{L^{\infty}(\R^{N})}}{\ell_{0}}\leq \frac{1}{2}.
\end{equation}
Fix $R>0$ and let $\psi_{R}\in C^{\infty}(\R^{N})$ be a function such that $\psi_{R}=0$ in $B_{R/2}$, $\psi_{R}=1$ in $B_{R}^{c}$, $\psi_{R}\in [0, 1]$ and $|\nabla \eta_{R}|\leq C/R$.
Since $\langle J'_{\e}(u_{n}), \eta_{R}u_{n}\rangle =o_{n}(1)$ we have
\begin{align*}
&\Re\Bigl(\iint_{\R^{2N}} \!\!\!\frac{(u_{n}(x)\!-\!u_{n}(y)e^{\imath A_{\e}(\frac{x+y}{2})\cdot (x-y)})\overline{((u_{n}\eta_{R})(x)\!-\!(u_{n}\eta_{R})(y)e^{\imath A_{\e}(\frac{x+y}{2})\cdot (x-y)})}}{|x-y|^{N+2s}} dx dy \Bigr)\\
&+\int_{\R^{N}} V_{\e}(x)\eta_{R} |u_{n}|^{2}\, dx=\int_{\R^{N}} \Bigl(\frac{1}{|x|^{\mu}}*G(\e x, |u_{n}|^{2})\Bigr) g(\e x, |u_{n}|^{2})u_{n}\psi_{R}+o_{n}(1).
\end{align*}
Taking into account 
\begin{align*}
&\Re\Bigl(\iint_{\R^{2N}}\!\!\! \frac{(u_{n}(x)\!-\!u_{n}(y)e^{\imath A_{\e}(\frac{x+y}{2})\cdot (x-y)})\overline{((u_{n}\eta_{R})(x)\!-\!(u_{n}\eta_{R})(y)e^{\imath A_{\e}(\frac{x+y}{2})\cdot (x-y)})}}{|x-y|^{N+2s}} dx dy \Bigr)\\
&=\Re\Bigl(\iint_{\R^{2N}} \!\!\!\overline{u_{n}(y)}e^{-\imath A_{\e}(\frac{x+y}{2})\cdot (x-y)}\frac{(u_{n}(x)\!-\!u_{n}(y)e^{\imath A_{\e}(\frac{x+y}{2})\cdot (x-y)})(\eta_{R}(x)\!-\!\eta_{R}(y))}{|x-y|^{N+2s}}  dx dy\Bigr)\\
&+\iint_{\R^{2N}} \!\!\!\eta_{R}(x)\frac{|u_{n}(x)\!-\!u_{n}(y)e^{\imath A_{\e}(\frac{x+y}{2})\cdot (x-y)}|^{2}}{|x-y|^{N+2s}} dx dy,
\end{align*}
and choosing $R>0$ large enough such that $\Lambda_{\e}\subset B_{\frac{R}{2}}$, we can use $(g_3)$-$(ii)$ and \eqref{Kbound} to get
\begin{align}\label{PS1}
&\iint_{\R^{2N}} \eta_{R}(x)\frac{|u_{n}(x)-u_{n}(y)e^{\imath A_{\e}(\frac{x+y}{2})\cdot (x-y)}|^{2}}{|x-y|^{N+2s}}\, dx dy+\int_{\R^{N}} V_{\e}(x)\eta_{R} |u_{n}|^{2}\, dx\nonumber\\
&\leq -\Re\Bigl(\iint_{\R^{2N}} \!\!\!\overline{u_{n}(y)}e^{-\imath A_{\e}(\frac{x+y}{2})\cdot (x-y)}\frac{(u_{n}(x)\!-\!u_{n}(y)e^{\imath A_{\e}(\frac{x+y}{2})\cdot (x-y)})(\eta_{R}(x)\!-\!\eta_{R}(y))}{|x-y|^{N+2s}}  dx dy\Bigr) \nonumber\\
&+\frac{1}{2}\int_{\R^{N}} V_{\e}(x) |u_{n}|^{2} \eta_{R} \, dx+o_{n}(1).
\end{align}
From the H\"older inequality and the boundedness of $(u_{n})$ in $\h$ it follows that
\begin{align}\label{PS2}
&\Bigl|\Re\Bigl(\iint_{\R^{2N}} \!\!\!\overline{u_{n}(y)}e^{-\imath A_{\e}(\frac{x+y}{2})\cdot (x-y)}\frac{(u_{n}(x)\!-\!u_{n}(y)e^{\imath A_{\e}(\frac{x+y}{2})\cdot (x-y)})(\eta_{R}(x)\!-\!\eta_{R}(y))}{|x-y|^{N+2s}}  dx dy\Bigr)\Bigr| \nonumber\\
&\leq \Bigl(\iint_{\R^{2N}} \frac{|u_{n}(x)-u_{n}(y)e^{\imath A_{\e}(\frac{x+y}{2})\cdot (x-y)}|^{2}}{|x-y|^{N+2s}}dxdy  \Bigr)^{\frac{1}{2}} \times \nonumber\\
&\quad \times \Bigl(\iint_{\R^{2N}} |\overline{u_{n}(y)}|^{2}\frac{|\eta_{R}(x)-\eta_{R}(y)|^{2}}{|x-y|^{N+2s}}  dxdy\Bigr)^{\frac{1}{2}} \nonumber\\
&\leq C \Bigl(\iint_{\R^{2N}} |u_{n}(y)|^{2}\frac{|\eta_{R}(x)-\eta_{R}(y)|^{2}}{|x-y|^{N+2s}} \, dxdy\Bigr)^{\frac{1}{2}}.
\end{align}
By using Lemma 2.1 in \cite{A6} we can see that
\begin{equation}\label{PS3}
\lim_{R\rightarrow \infty}\limsup_{n\rightarrow \infty} \iint_{\R^{2N}} |u_{n}(y)|^{2}\frac{|\eta_{R}(x)-\eta_{R}(y)|^{2}}{|x-y|^{N+2s}} \, dxdy=0.
\end{equation}
Then, putting together \eqref{PS1}, \eqref{PS2} and \eqref{PS3} we can deduce that \eqref{DF} holds true.

\noindent
{\bf Step $2$}: Let us prove that $u_{n}\rightarrow u$ in $H^{s}_{\e}$ as $n\rightarrow \infty$.\\
Since $u_{n}\rightharpoonup u$ in $\h$ and $\langle J'_{\e}(u_{n}),u_{n}\rangle=\langle J'_{\e}(u_{n}),u\rangle=o_{n}(1)$
we can note that
$$
\|u_{n}\|^{2}_{\e}-\|u\|^{2}_{\e}=\|u_{n}-u\|^{2}_{\e}+o_{n}(1)=\int_{\R^{N}} \tilde{K}_{\e}(u_{n})g_{\e}(x, |u_{n}|^{2})(|u_{n}|^{2}-|u|^{2})dx+o_{n}(1).
$$
Therefore, being $\h$ be a Hilbert space, it is enough to show that 
$$
\int_{\R^{N}} \tilde{K}_{\e}(u_{n})g_{\e}(x, |u_{n}|^{2})(|u_{n}|^{2}-|u|^{2})dx=o_{n}(1).
$$
By Lemma \ref{lemK} we know that $|\tilde{K}_{\e}(u_{n})|\leq C$ for all $n\in \N$. Since $|u_{n}|\rightarrow |u|$ in $L^{r}(B_{R})$ for all $r\in [2, 2^{*}_{s})$ and $R>0$, we obtain 
\begin{align}
\left|\int_{B_{R}} \tilde{K}_{\e}(u_{n})g_{\e}(x, |u_{n}|^{2})(|u_{n}|^{2}-|u|^{2})dx\right|\leq C\int_{B_{R}} |g_{\e}(x, |u_{n}|^{2})(|u_{n}|^{2}-|u|^{2})|dx\rightarrow 0.
\end{align}
By the Step $1$ and Theorem \ref{Sembedding}, for any $\eta>0$ there exists $R_{\eta}>0$ such that
$$
\limsup_{n\rightarrow \infty} \int_{\R^{N}\setminus B_{R}} \tilde{K}_{\e}(u_{n}) |g(\e x, |u_{n}|^{2})|u_{n}|^{2}| \, dx\leq C\eta. 
$$
In similar way, from H\"older inequality, we can see that
$$
\limsup_{n\rightarrow \infty} \int_{\R^{N}\setminus B_{R}} \tilde{K}_{\e}(u_{n}) |g(\e x, |u_{n}|^{2})|u|^{2}| \, dx\leq C\eta. 
$$
Taking into account the above limits we can infer that  
$$
\lim_{n\rightarrow \infty} \int_{\R^{N}} \tilde{K}_{\e}(u_{n}) g(\e x, |u_{n}|^{2})(|u_{n}|^{2}-|u|^{2}) \, dx=0. 
$$
This ends the proof of Lemma \ref{PSc}.
\end{proof}

\section{Concentration of solutions to \eqref{P}}
In this section we give the proof of the main result of this paper. 
Firstly, we consider the limit problem associated with \eqref{R}, that is
\begin{equation}\label{APe}
(-\Delta)^{s} u + V_{0}u = \left(\frac{1}{|x|^{\mu}}*F(|u|^{2})\right)f(|u|^{2})u \mbox{ in } \R^{N}, 
\end{equation}
and the corresponding energy functional $J_{0}: H^{s}_{0}\rightarrow \R$ given by
$$
J_{0}(u)=\frac{1}{2}\|u\|^{2}_{V_{0}}-\mathfrak{F}(u), 
$$
where $H_{0}^{s}$ is the space $H^{s}(\R^{N}, \R)$ endowed with the norm 
$$
\|u\|^{2}_{V_{0}}=[u]^{2}+\int_{\R^{N}} V_{0} u^{2}\,dx,
$$
and
$$
\mathfrak{F}(u)=\frac{1}{4}\int_{\R^{N}} \left(\frac{1}{|x|^{\mu}}*F(|u|^{2})\right)F(|u|^{2})\, dx.
$$
As in the previous section, it is easy to see that $J_{0}$ has a mountain pass geometry and we denote by $c_{V_{0}}$ the mountain pass level of the functional $J_{0}$.

Let us introduce the Nehari manifold associated with (\ref{Pe}), that is
\begin{equation*}
\mathcal{N}_{\e}:= \{u\in \h \setminus \{0\} : \langle J_{\e}'(u), u \rangle =0\},
\end{equation*}
and we denote by $\mathcal{N}_{0}$ the Nehari manifold associated with \eqref{APe}.
It is standard to verify (see \cite{W}) that $c_{\e}$ can be characterized as 
$$
c_{\e}=\inf_{u\in \h\setminus\{0\}} \sup_{t\geq 0} J_{\e}(t u)=\inf_{u\in \mathcal{N}_{\e}} J_{\e}(u).
$$
In the next result we stress an interesting relation between $c_{\e}$ and $c_{V_{0}}$.
\begin{lemma}\label{AMlem1}
The numbers $c_{\e}$ and $c_{V_{0}}$ satisfy the following inequality
$$
\limsup_{\e\rightarrow 0} c_{\e}\leq c_{V_{0}}.
$$
\end{lemma}
\begin{proof}
In view of Lemma $3.3$ in \cite{Apota}, there exists a ground state $w\in H^{s}(\R^{N}, \R)$ to the autonomous problem \eqref{APe}, so that $J'_{0}(w)=0$ and $J_{0}(w)=c_{V_{0}}$. Moreover, we know that $w\in C^{0, \mu}(\R^{N})$ and $w>0$ in $\R^{N}$.  
In what follows, we show that $w$ satisfies the following useful estimate:
\begin{equation}\label{remdecay}
0<w(x)\leq \frac{C}{|x|^{N+2s}} \mbox{ for large } |x|.
\end{equation}
By using $(f_1)$, $\lim_{|x|\rightarrow\infty}w(x)=0$ and the boundedness of the convolution term (see proof of Lemma \ref{lemK}) we can find $R>0$ such that $\left(\frac{1}{|x|^{\mu}}*F(w^{2})\right)f(w^{2})w\leq \frac{V_{0}}{2}w$ in $B_{R}^{c}$. In particular we have
\begin{equation}\label{BBMP1}
(-\Delta)^{s}w+\frac{V_{0}}{2}w=\left(\frac{1}{|x|^{\mu}}*F(w^{2})\right)f(w^{2})w-\left(V_{0}-\frac{V_{0}}{2}\right)w\leq 0 \mbox{ in } B_{R}^{c}.
\end{equation}
In view of Lemma $4.2$ in \cite{FQT} and by rescaling, we know that there exists a positive function $w_{1}$ and a constant $C_{1}>0$ such that for large $|x|>R$ it holds that $w_{1}(x)=C_{1}|x|^{-(N+2s)}$ and 
\begin{equation}\label{BBMP2}
(-\Delta)^{s}w_{1}+\frac{V_{0}}{2}w_{1}\geq 0 \mbox{ in }  B^{c}_{R}.
\end{equation}
Taking into account the continuity of $w$ and $w_{1}$ there exists $C_{2}>0$ such that $w_{2}(x)=w(x)-C_{2}w_{1}(x)\leq 0$ on $|x|=R$ (taking $R$ larger if necessary). Moreover, we can see that $(-\Delta)^{s}w_{2}+\frac{V_{0}}{2}w_{2}\leq 0$ for $|x|\geq R$ and by using the maximum principle we can infer that $w_{2}\leq 0$ in $B_{R}^{c}$, that is $w\leq C_{2}w_{1}$ in $B_{R}^{c}$. This fact implies that \eqref{remdecay} holds true.

Now, fix a cut-off function $\eta\in C^{\infty}_{c}(\R^{N}, [0,1])$ such that $\eta=1$ in a neighborhood of zero $B_{\frac{\delta}{2}}$ and $\supp(\eta)\subset B_{\delta}\subset \Lambda$ for some $\delta>0$. 
Let us define $w_{\e}(x):=\eta_{\e}(x)w(x) e^{\imath A(0)\cdot x}$, with $\eta_{\e}(x)=\eta(\e x)$ for $\e>0$, and we observe that $|w_{\e}|=\eta_{\e}w$ and $w_{\e}\in \h$ in view of Lemma \ref{aux}. Let us prove that
\begin{equation}\label{limwr}
\lim_{\e\rightarrow 0}\|w_{\e}\|^{2}_{\e}=\|w\|_{0}^{2}\in(0, \infty).
\end{equation}
Clearly, $\int_{\R^{N}} V_{\e}(x)|w_{\e}|^{2}dx\rightarrow \int_{\R^{N}} V_{0} |w|^{2}dx$. Now, we show that
\begin{equation}\label{limwr*}
\lim_{\e\rightarrow 0}[w_{\e}]^{2}_{A_{\e}}=[w]^{2}.
\end{equation}
We note that, in view of Lemma $5$ in \cite{PP}, we have 
\begin{equation}\label{PPlem}
[\eta_{\e} w]\rightarrow [w] \mbox{ as } \e\rightarrow 0.
\end{equation}
On the other hand
\begin{align*}
&[w_{\e}]_{A_{\e}}^{2}\nonumber \\
&=\iint_{\R^{2N}} \frac{|e^{\imath A(0)\cdot x}\eta_{\e}(x)w(x)-e^{\imath A_{\e}(\frac{x+y}{2})\cdot (x-y)}e^{\imath A(0)\cdot y} \eta_{\e}(y)w(y)|^{2}}{|x-y|^{N+2s}} dx dy \nonumber \\
&=[\eta_{\e} w]^{2}
+\iint_{\R^{2N}} \frac{\eta_{\e}^2(y)w^2(y) |e^{\imath [A_{\e}(\frac{x+y}{2})-A(0)]\cdot (x-y)}-1|^{2}}{|x-y|^{N+2s}} dx dy\\
&+2\Re \iint_{\R^{2N}} \frac{(\eta_{\e}(x)w(x)-\eta_{\e}(y)w(y))\eta_{\e}(y)w(y)(1-e^{-\imath [A_{\e}(\frac{x+y}{2})-A(0)]\cdot (x-y)})}{|x-y|^{N+2s}} dx dy \\
&=: [\eta_{\e} w]^{2}+X_{\e}+2Y_{\e}.
\end{align*}
Taking into account
$|Y_{\e}|\leq [\eta_{\e} w] \sqrt{X_{\e}}$ and \eqref{PPlem}, we need to prove that $X_{\e}\rightarrow 0$ as $\e\rightarrow 0$ to deduce that \eqref{limwr*} holds true.

Let us observe that for $0<\beta<\alpha/({1+\alpha-s})$ we get
\begin{equation}\label{Ye}
\begin{split}
X_{\e}
&\leq \int_{\R^{N}} w^{2}(y) dy \int_{|x-y|\geq\e^{-\beta}} \frac{|e^{\imath [A_{\e}(\frac{x+y}{2})-A(0)]\cdot (x-y)}-1|^{2}}{|x-y|^{N+2s}} dx\\
&+\int_{\R^{N}} w^{2}(y) dy  \int_{|x-y|<\e^{-\beta}} \frac{|e^{\imath [A_{\e}(\frac{x+y}{2})-A(0)]\cdot (x-y)}-1|^{2}}{|x-y|^{N+2s}} dx \\
&=:X^{1}_{\e}+X^{2}_{\e}.
\end{split}
\end{equation}
Since $|e^{\imath t}-1|^{2}\leq 4$ and $w\in H^{s}(\R^{N}, \R)$, we can see that
\begin{equation}\label{Ye1}
X_{\e}^{1}\leq C \int_{\R^{N}} w^{2}(y) dy \int_{\e^{-\beta}}^\infty \rho^{-1-2s} d\rho\leq C \e^{2\beta s} \rightarrow 0.
\end{equation}
Now, by using $|e^{\imath t}-1|^{2}\leq t^{2}$ for all $t\in \R$, $A\in C^{0,\alpha}(\R^N,\R^N)$ for $\alpha\in(0,1]$, and $|x+y|^{2}\leq 2(|x-y|^{2}+4|y|^{2})$, we obtain
\begin{equation}\label{Ye2}
\begin{split}
X^{2}_{\e}&
	\leq \int_{\R^{N}} w^{2}(y) dy  \int_{|x-y|<\e^{-\beta}} \frac{|A_{\e}\left(\frac{x+y}{2}\right)-A(0)|^{2} }{|x-y|^{N+2s-2}} dx \\
	&\leq C\e^{2\alpha} \int_{\R^{N}} w^{2}(y) dy  \int_{|x-y|<\e^{-\beta}} \frac{|x+y|^{2\alpha} }{|x-y|^{N+2s-2}} dx \\
	&\leq C\e^{2\alpha} \left(\int_{\R^{N}} w^{2}(y) dy  \int_{|x-y|<\e^{-\beta}} \frac{1 }{|x-y|^{N+2s-2-2\alpha}} dx\right.\\
	&\qquad\qquad+ \left. \int_{\R^{N}} |y|^{2\alpha} w^{2}(y) dy  \int_{|x-y|<\e^{-\beta}} \frac{1}{|x-y|^{N+2s-2}} dx\right) \\
	&=: C\e^{2\alpha} (X^{2, 1}_{\e}+X^{2, 2}_{\e}).
	\end{split}
	\end{equation}	
	Then
	\begin{equation}\label{Ye21}
	X^{2, 1}_{\e}
	= C  \int_{\R^{N}} w^{2}(y) dy \int_0^{\e^{-\beta}} \rho^{1+2\alpha-2s} d\rho
	\leq C\e^{-2\beta(1+\alpha-s)}.
	\end{equation}
	On the other hand, using \eqref{remdecay}, we can infer that
	\begin{equation}\label{Ye22}
	\begin{split}
	 X^{2, 2}_{\e}
	 &\leq C  \int_{\R^{N}} |y|^{2\alpha} w^{2}(y) dy \int_0^{\e^{-\beta}}\rho^{1-2s} d\rho  \\
	&\leq C \e^{-2\beta(1-s)} \left[\int_{B_1(0)}  w^{2}(y) dy + \int_{B_1^c(0)} \frac{1}{|y|^{2(N+2s)-2\alpha}} dy \right]  \\
	&\leq C \e^{-2\beta(1-s)}.
	\end{split}
	\end{equation}
	Taking into account \eqref{Ye}, \eqref{Ye1}, \eqref{Ye2}, \eqref{Ye21} and \eqref{Ye22} we have $X_{\e}\rightarrow 0$, and then \eqref{limwr} holds.
	
Now, let $t_{\e}>0$ be the unique number such that 
\begin{equation*}
J_{\e}(t_{\e} w_{\e})=\max_{t\geq 0} J_{\e}(t w_{\e}).
\end{equation*}
As a consequence, $t_{\e}$ satisfies 
\begin{align}\label{AS1}
\|w_{\e}\|_{\e}^{2}&=\int_{\R^{N}} \left(\frac{1}{|x|^{\mu}}*G(\e x, t_{\e}^{2} |w_{\e}|^{2})\right)g(\e x, t_{\e}^{2} |w_{\e}|^{2}) |w_{\e}|^{2}dx \nonumber \\
&=\int_{\R^{N}} \left(\frac{1}{|x|^{\mu}}*F(t_{\e}^{2} |w_{\e}|^{2})\right) f(t_{\e}^{2} |w_{\e}|^{2}) |w_{\e}|^{2}dx
\end{align}
where we used $supp(\eta)\subset \Lambda$ and $g=f$ on $\Lambda$.\\
Let us prove that $t_{\e}\rightarrow 1$ as $\e\rightarrow 0$. Since $\eta=1$ in $B_{\frac{\delta}{2}}$, $w$ is a continuous positive function, and recalling that $f(t)$ and $F(t)/t$ are both increasing, we have 
$$
\|w_{\e}\|_{\e}^{2}\geq \frac{F(t_{\e}^{2}\alpha_{0}^{2})}{\alpha_{0}^{2}} f(t_{\e}^{2}\alpha^{2}_{0})\int_{B_{\frac{\delta}{2}}}\int_{B_{\frac{\delta}{2}}} 
\frac{|w(x)|^{2} |w(y)|^{2}}{|x-y|^{\mu}}dx dy, 
$$
where $\alpha_{0}=\min_{\bar{B}_{\frac{\delta}{2}}} w>0$. \\
Let us prove that $t_{\e}\rightarrow t_{0}\in (0, \infty)$ as $\e\rightarrow 0$. Indeed, if $t_{\e}\rightarrow \infty$ as $\e\rightarrow 0$ then we can use $(f_3)$ to deduce that $\|w\|_{0}^{2}= \infty$ which gives a contradiction due to \eqref{limwr}.
When $t_{\e}\rightarrow 0$ as $\e\rightarrow 0$ we can use $(f_1)$ to infer that $\|w\|_{0}^{2}= 0$ which is impossible in view of \eqref{limwr}.\\
Then, taking the limit as $\e\rightarrow 0$ in \eqref{AS1} and using \eqref{limwr}, we can deduce that 
\begin{equation}\label{AS2}
\|w\|_{0}^{2}=\int_{\R^{N}} \left(\frac{1}{|x|^{\mu}}*F(t_{0}^{2}|w|^{2})\right)f(t_{0}^{2} |w|^{2}) |w|^{2}dx.
\end{equation}
Since $w\in \mathcal{N}_{0}$ and using $(f_3)$, we obtain $t_{0}=1$. Hence, from the Dominated Convergence Theorem, we can see that $\lim_{\e\rightarrow 0} J_{\e}(t_{\e} w_{\e})=J_{0}(w)=c_{V_{0}}$.
Recalling that $c_{\e}\leq \max_{t\geq 0} J_{\e}(t w_{\e})=J_{\e}(t_{\e} w_{\e})$, we can infer  that
$\limsup_{\e\rightarrow 0} c_{\e}\leq c_{V_{0}}$.
\end{proof}

\noindent
Arguing as in \cite{Apota}, we can deduce the following result for the autonomous problem:
\begin{lemma}\label{FS}
Let $(u_{n})\subset \mathcal{N}_{0}$ be a sequence satisfying $J_{0}(u_{n})\rightarrow c_{V_{0}}$. Then, up to subsequences, the following alternatives holds:
\begin{enumerate}
\item [$(i)$] $(u_{n})$ strongly converges in $H^{s}(\R^{N}, \R)$, 
\item [$(ii)$] there exists a sequence $(\tilde{y}_{n})\subset \R^{N}$ such that,  up to a subsequence, $v_{n}(x)=u_{n}(x+\tilde{y}_{n})$ converges strongly in $H^{s}(\R^{N}, \R)$.
\end{enumerate}
In particular, there exists a minimizer for $c_{V_{0}}$.
\end{lemma}

\noindent
Now, we prove the following useful compactness result.
\begin{lemma}\label{prop3.3}
Let $\e_{n}\rightarrow 0$ and $(u_{n})\subset H^{s}_{\e_{n}}$ such that $J_{\e_{n}}(u_{n})=c_{\e_{n}}$ and $J'_{\e_{n}}(u_{n})=0$. Then there exists $(\tilde{y}_{n})\subset \R^{N}$ such that $v_{n}(x)=|u_{n}|(x+\tilde{y}_{n})$ has a convergent subsequence in $H^{s}(\R^{N}, \R)$. Moreover, up to a subsequence, $y_{n}=\e_{n} \tilde{y}_{n}\rightarrow y_{0}$ for some $y_{0}\in \Lambda$ such that $V(y_{0})=V_{0}$.
\end{lemma}
\begin{proof}
Taking into account $\langle J'_{\e_{n}}(u_{n}), u_{n}\rangle=0$, $J_{\e_{n}}(u_{n})= c_{\e_{n}}$, Lemma \ref{AMlem1} and arguing as in Lemma \ref{PSc}, it is easy to see that $(u_{n})$ is bounded in $H^{s}_{\e_{n}}$ and $\|u_{n}\|^{2}_{\e_{n}}\leq 4(\kappa+1)$ for all $n\in \N$. Moreover, from Lemma \ref{DI}, we also know that $(|u_{n}|)$ is bounded in $H^{s}(\R^{N}, \R)$.\\
Let us prove that there exist a sequence $(\tilde{y}_{n})\subset \R^{N}$ and constants $R>0$ and $\gamma>0$ such that
\begin{equation}\label{sacchi}
\liminf_{n\rightarrow \infty}\int_{B_{R}(\tilde{y}_{n})} |u_{n}|^{2} \, dx\geq \gamma>0.
\end{equation}
Otherwise, if \eqref{sacchi} does not hold, then for all $R>0$ we have
$$
\lim_{n\rightarrow \infty}\sup_{y\in \R^{N}}\int_{B_{R}(y)} |u_{n}|^{2} \, dx=0.
$$
From the boundedness $(|u_{n}|)$ and Lemma $2.2$ in \cite{FQT} we can see that $|u_{n}|\rightarrow 0$ in $L^{q}(\R^{N}, \R)$ for any $q\in (2, 2^{*}_{s})$. 
By using $(g_1)$-$(g_2)$ and Lemma \ref{lemK} we  can deduce that
\begin{align}\label{glimiti}
\lim_{n\rightarrow \infty}\int_{\R^{N}} \tilde{K}_{\e_{n}}(u_{n})  g(\e_{n} x, |u_{n}|^{2}) |u_{n}|^{2} \,dx=0= \lim_{n\rightarrow \infty}\int_{\R^{N}} \tilde{K}_{\e_{n}}(u_{n}) G(\e_{n} x, |u_{n}|^{2}) \, dx.
\end{align}
Since $\langle J'_{\e_{n}}(u_{n}), u_{n}\rangle=0$, we can use \eqref{glimiti} to deduce that $\|u_{n}\|_{\e_{n}}\rightarrow 0$ as $n\rightarrow \infty$. This gives a contradiction because $u_{n}\in \mathcal{N}_{\e_{n}}$ and by using $(g_1)$, $(g_2)$ and Lemma \ref{lemK} we can find $\alpha_{0}>0$ such that $\|u_{n}\|^{2}_{\e_{n}}\geq \alpha_{0}$ for all $n\in \mathbb{N}$.\\
Set $v_{n}(x)=|u_{n}|(x+\tilde{y}_{n})$. Then $(v_{n})$ is bounded in $H^{s}(\R^{N}, \R)$, and we may assume that 
$v_{n}\rightharpoonup v\not\equiv 0$ in $H^{s}(\R^{N}, \R)$  as $n\rightarrow \infty$.
Fix $t_{n}>0$ such that $\tilde{v}_{n}=t_{n} v_{n}\in \mathcal{N}_{0}$. By using Lemma \ref{DI}, we can see that 
$$
c_{V_{0}}\leq J_{0}(\tilde{v}_{n})\leq \max_{t\geq 0}J_{\e_{n}}(tv_{n})= J_{\e_{n}}(u_{n})=c_{\e_{n}}
$$
which together with Lemma \ref{AMlem1} implies that $J_{0}(\tilde{v}_{n})\rightarrow c_{V_{0}}$. In particular, $\tilde{v}_{n}\nrightarrow 0$ in $H^{s}(\R^{N}, \R)$.
Since $(v_{n})$ and $(\tilde{v}_{n})$ are bounded in $H^{s}(\R^{N}, \R)$ and $\tilde{v}_{n}\nrightarrow 0$  in $H^{s}(\R^{N}, \R)$, we obtain that $t_{n}\rightarrow t^{*}> 0$. 
From the uniqueness of the weak limit, we can deduce that $\tilde{v}_{n}\rightharpoonup \tilde{v}=t^{*}v\not\equiv 0$ in $H^{s}(\R^{N}, \R)$. 
This together with Lemma \ref{FS} gives
\begin{equation}\label{elena}
\tilde{v}_{n}\rightarrow \tilde{v} \mbox{ in } H^{s}(\R^{N}, \R),
\end{equation} 
and as a consequence $v_{n}\rightarrow v$ in $H^{s}(\R^{N}, \R)$ as $n\rightarrow \infty$.

Now, we set $y_{n}=\e_{n}\tilde{y}_{n}$. We aim to prove that $(y_{n})$ admits a subsequence, still denoted by $y_{n}$, such that $y_{n}\rightarrow y_{0}$ for some $y_{0}\in \Lambda$ such that $V(y_{0})=V_{0}$. Firstly, we prove that $(y_{n})$ is bounded. Assume by contradiction that, up to a subsequence, $|y_{n}|\rightarrow \infty$ as $n\rightarrow \infty$. Take $R>0$ such that $\Lambda \subset B_{R}(0)$. Since we may suppose that  $|y_{n}|>2R$, we have that $|\e_{n}z+y_{n}|\geq |y_{n}|-|\e_{n}z|>R$ for any $z\in B_{R/\e_{n}}$. 
Taking into account $(u_{n})\subset \mathcal{N}_{\e_{n}}$, $(V_{1})$, Lemma \ref{DI} and the change of variable $x\mapsto z+\tilde{y}_{n}$ we get
\begin{align*}
&[v_{n}]^{2}+\int_{\R^{N}} V_{0} v_{n}^{2}\, dx \\
&\leq C_{0}\int_{\R^{N}}  g(\e_{n} z+y_{n}, |v_{n}|^{2}) |v_{n}|^{2} \, dz \\
&\leq C_{0}\int_{B_{\frac{R}{\e_{n}}}(0)}  \tilde{f}(|v_{n}|^{2}) |v_{n}|^{2} \, dz+C_{0}\int_{\R^{N}\setminus B_{\frac{R}{\e_{n}}}(0)}  f(|v_{n}|^{2}) |v_{n}|^{2} \, dz,
\end{align*}
where we used  $u_{n}\in \B$ for all $n$ big enough and Lemma \ref{lemK}.
By using $v_{n}\rightarrow v$ in $H^{s}(\R^{N}, \R)$ as $n\rightarrow \infty$ and $\tilde{f}(t)\leq \frac{V_{0}}{\ell_{0}}$ we obtain
$$
\min\left\{1, \frac{V_{0}}{2} \right\} \left([v_{n}]^{2}+\int_{\R^{N}} |v_{n}|^{2}\, dx\right)=o_{n}(1).
$$
Then $v_{n}\rightarrow 0$ in $H^{s}(\R^{N}, \R)$ and this is impossible. Therefore, $(y_{n})$ is bounded and we may assume that $y_{n}\rightarrow y_{0}\in \R^{N}$. If $y_{0}\notin \overline{\Lambda}$, we can argue as before to deduce that $v_{n}\rightarrow 0$ in $H^{s}(\R^{N}, \R)$, which gives a contradiction. Therefore $y_{0}\in \overline{\Lambda}$, and in view of $(V_2)$, it is enough to verify that $V(y_{0})=V_{0}$ to conclude the proof of lemma. Assume by contradiction that $V(y_{0})>V_{0}$.

Then, by using (\ref{elena}), Fatou's Lemma, the invariance of  $\R^{N}$ by translations, Lemma \ref{DI} and Lemma \ref{AMlem1}, we get 
\begin{align*}
&c_{V_{0}}=J_{0}(\tilde{v})<\frac{1}{2}[\tilde{v}]^{2}+\frac{1}{2}\int_{\R^{N}} V(y_{0})\tilde{v}^{2} \, dx-\mathfrak{F}(\tilde{v}) \\
&\leq \liminf_{n\rightarrow \infty}\Bigl[\frac{1}{2}[\tilde{v}_{n}]^{2}+\frac{1}{2}\int_{\R^{N}} V(\e_{n}z+y_{n}) |\tilde{v}_{n}|^{2} \, dz-\mathfrak{F}(\tilde{v}_{n}) \Bigr] \\ 
&\leq \liminf_{n\rightarrow \infty}\Bigl[\frac{t_{n}^{2}}{2}[|u_{n}|]^{2}+\frac{t_{n}^{2}}{2}\int_{\R^{N}} V(\e_{n}z) |u_{n}|^{2} \, dz-\mathfrak{F}(t_{n} u_{n})  \Bigr] \\
&\leq \liminf_{n\rightarrow \infty} J_{\e_{n}}(t_{n} u_{n}) \leq \liminf_{n\rightarrow \infty} J_{\e_{n}}(u_{n})\leq c_{V_{0}}
\end{align*}
which gives a contradiction.
\end{proof}

\noindent
The next lemma will be fundamental to prove that the solutions of \eqref{Pe} are also solutions  of the original problem \eqref{P}. We will use a suitable variant of the Moser iteration argument \cite{Moser}.
\begin{lemma}\label{moser} 
Let $\e_{n}\rightarrow 0$ and $u_{n}\in H^{s}_{\e_{n}}$ be a solution to \eqref{Pe}. 
Then $v_{n}=|u_{n}|(\cdot+\tilde{y}_{n})$ satisfies $v_{n}\in L^{\infty}(\R^{N},\R)$ and there exists $C>0$ such that 
$$
\|v_{n}\|_{L^{\infty}(\R^{N})}\leq C \mbox{ for all } n\in \mathbb{N},
$$
where $\tilde{y}_{n}$ is given by Lemma \ref{prop3.3}.
Moreover
$$
\lim_{|x|\rightarrow \infty} v_{n}(x)=0 \mbox{ uniformly in } n\in \mathbb{N}.
$$
\end{lemma}
\begin{proof}
For any $L>0$ we define $u_{L,n}:=\min\{|u_{n}|, L\}\geq 0$ and we set $v_{L, n}=u_{L,n}^{2(\beta-1)}u_{n}$ where $\beta>1$ will be chosen later.
Taking $v_{L, n}$ as test function in (\ref{Pe}) we can see that
\begin{align}\label{conto1FF}
&\Re\Bigl(\iint_{\R^{2N}} \frac{(u_{n}(x)-u_{n}(y)e^{\imath A(\frac{x+y}{2})\cdot (x-y)})}{|x-y|^{N+2s}} \times \nonumber\\
&\quad \times \overline{((u_{n}u_{L,n}^{2(\beta-1)})(x)-(u_{n}u_{L,n}^{2(\beta-1)})(y)e^{\imath A(\frac{x+y}{2})\cdot (x-y)})}  dx dy\Bigr)   \nonumber \\
&=\int_{\R^{N}} \tilde{K}_{\e}(u_{n}) g(\e_{n} x, |u_{n}|^{2}) |u_{n}|^{2}u_{L,n}^{2(\beta-1)}  \,dx-\int_{\R^{N}} V(\e_{n} x) |u_{n}|^{2} u_{L,n}^{2(\beta-1)} \, dx.
\end{align}
Let us observe that
\begin{align*}
&\Re\left[(u_{n}(x)-u_{n}(y)e^{\imath A(\frac{x+y}{2})\cdot (x-y)})\overline{(u_{n}u_{L,n}^{2(\beta-1)}(x)-u_{n}u_{L,n}^{2(\beta-1)}(y)e^{\imath A(\frac{x+y}{2})\cdot (x-y)})}\right] \\
&=\Re\Bigl[|u_{n}(x)|^{2}v_{L}^{2(\beta-1)}(x)-u_{n}(x)\overline{u_{n}(y)} u_{L,n}^{2(\beta-1)}(y)e^{-\imath A(\frac{x+y}{2})\cdot (x-y)}\\
&\quad-u_{n}(y)\overline{u_{n}(x)} u_{L,n}^{2(\beta-1)}(x) e^{\imath A(\frac{x+y}{2})\cdot (x-y)} +|u_{n}(y)|^{2}u_{L,n}^{2(\beta-1)}(y) \Bigr] \\
&\geq (|u_{n}(x)|^{2}u_{L,n}^{2(\beta-1)}(x)-|u_{n}(x)||u_{n}(y)|u_{L,n}^{2(\beta-1)}(y) \\
&\quad -|u_{n}(y)||u_{n}(x)|u_{L,n}^{2(\beta-1)}(x)+|u_{n}(y)|^{2}u^{2(\beta-1)}_{L,n}(y) \\
&=(|u_{n}(x)|-|u_{n}(y)|)(|u_{n}(x)|u_{L,n}^{2(\beta-1)}(x)-|u_{n}(y)|u_{L,n}^{2(\beta-1)}(y)),
\end{align*}
which implies that
\begin{align}\label{realeF}
&\Re\Bigl(\iint_{\R^{2N}} \frac{(u_{n}(x)-u_{n}(y)e^{\imath A(\frac{x+y}{2})\cdot (x-y)})}{|x-y|^{N+2s}} \times \nonumber\\
&\quad \times \overline{((u_{n}u_{L,n}^{2(\beta-1)})(x)-(u_{n}u_{L,n}^{2(\beta-1)})(y)e^{\imath A(\frac{x+y}{2})\cdot (x-y)})} dx dy\Bigr) \nonumber\\
&\geq \iint_{\R^{2N}} \frac{(|u_{n}(x)|-|u_{n}(y)|)}{|x-y|^{N+2s}} (|u_{n}(x)|u_{L,n}^{2(\beta-1)}(x)-|u_{n}(y)|u_{L,n}^{2(\beta-1)}(y))\, dx dy.
\end{align}
As in \cite{Apota}, for all $t\geq 0$, we define
\begin{equation*}
\gamma(t)=\gamma_{L, \beta}(t)=t t_{L}^{2(\beta-1)}
\end{equation*}
where  $t_{L}=\min\{t, L\}$. 
Since $\gamma$ is an increasing function we have
\begin{align*}
(a-b)(\gamma(a)- \gamma(b))\geq 0 \quad \mbox{ for any } a, b\in \R.
\end{align*}
Let 
\begin{equation*}
\Lambda(t)=\frac{|t|^{2}}{2} \quad \mbox{ and } \quad \Gamma(t)=\int_{0}^{t} (\gamma'(\tau))^{\frac{1}{2}} d\tau. 
\end{equation*}
Since
\begin{equation}\label{Gg}
\Lambda'(a-b)(\gamma(a)-\gamma(b))\geq |\Gamma(a)-\Gamma(b)|^{2} \mbox{ for any } a, b\in\R, 
\end{equation}
we get
\begin{align}\label{Gg1}
|\Gamma(|u_{n}(x)|)- \Gamma(|u_{n}(y)|)|^{2} \leq (|u_{n}(x)|- |u_{n}(y)|)((|u_{n}|u_{L,n}^{2(\beta-1)})(x)- (|u_{n}|u_{L,n}^{2(\beta-1)})(y)). 
\end{align}
Putting together \eqref{realeF} and \eqref{Gg1}, we can see that
\begin{align}\label{conto1FFF}
&\Re\Bigl(\iint_{\R^{2N}} \frac{(u_{n}(x)-u_{n}(y)e^{\imath A(\frac{x+y}{2})\cdot (x-y)})}{|x-y|^{N+2s}} \times \nonumber\\
&\quad \times \overline{(u_{n}u_{L,n}^{2(\beta-1)}(x)\!-\!u_{n}u_{L,n}^{2(\beta-1)}(y)e^{\imath A(\frac{x+y}{2})\cdot (x-y)})} dx dy\Bigr) 	\nonumber\\
&\geq [\Gamma(|u_{n}|)]^{2}.
\end{align}
Since $\Gamma(|u_{n}|)\geq \frac{1}{\beta} |u_{n}| u_{L,n}^{\beta-1}$ and using the fractional Sobolev embedding $\mathcal{D}^{s,2}(\R^{N}, \R)\subset L^{\2}(\R^{N}, \R)$ (see \cite{DPV}), we can infer that 
\begin{equation}\label{SS1}
[\Gamma(|u_{n}|)]^{2}\geq S_{*} \|\Gamma(|u_{n}|)\|^{2}_{L^{\2}(\R^{N})}\geq \left(\frac{1}{\beta}\right)^{2} S_{*}\||u_{n}| u_{L,n}^{\beta-1}\|^{2}_{L^{\2}(\R^{N})}.
\end{equation}
Then \eqref{conto1FF}, \eqref{conto1FFF} and \eqref{SS1} yield
\begin{align}\label{BMS}
&\left(\frac{1}{\beta}\right)^{2} S_{*}\||u_{n}| u_{L,n}^{\beta-1}\|^{2}_{L^{\2}(\R^{N})}+\int_{\R^{N}} V(\e_{n} x)|u_{n}|^{2}u_{L,n}^{2(\beta-1)} dx\nonumber\\
&\leq \int_{\R^{N}} \tilde{K}_{\e_{n}}(u_{n}) g(\e_{n}x, |u_{n}|^{2}) |u_{n}|^{2} u_{L,n}^{2(\beta-1)} dx.
\end{align}
By $(g_1)$ and $(g_2)$, we know that for any $\xi>0$ there exists $C_{\xi}>0$ such that
\begin{equation}\label{SS2}
g(x, t^{2})t^{2}\leq \xi |t|^{2}+C_{\xi}|t|^{\2} \mbox{ for any } (x, t)\in \R^{N}\times\R.
\end{equation}
Hence, using \eqref{BMS}, \eqref{SS2}, $u_{n}\in \B$, Lemma \ref{lemK} and choosing $\xi>0$ sufficiently small, we can see that
\begin{equation}\label{simo1}
\|w_{L,n}\|_{L^{\2}(\R^{N})}^{2}\leq C \beta^{2} \int_{\R^{N}} |u_{n}|^{q}u_{L,n}^{2(\beta-1)},
\end{equation}
for some $C$ independent of $\beta$, $L$ and $n$. Here we set $w_{L,n}:=|u_{n}| u_{L,n}^{\beta-1}$.
Arguing as in the proof of Lemma $5.1$ in \cite{Apota} we can see that
\begin{equation}\label{UBu}
\|u_{n}\|_{L^{\infty}(\R^{N})}\leq K \mbox{ for all } n\in \mathbb{N}.
\end{equation}
Moreover, by interpolation, $(|u_{n}|)$ strongly converges in $L^{r}(\R^{N}, \R)$ for all $r\in (2, \infty)$, and in view of the growth assumptions on $g$, also $g(\e_{n} x, |u_{n}|^{2})|u_{n}|$ strongly converges  in the same Lebesgue spaces. \\
In what follows, we show that $|u_{n}|$ is a weak subsolution to 
\begin{equation}\label{Kato0}
\left\{
\begin{array}{ll}
(-\Delta)^{s}v+V(\e_{n} x) v=\left(\frac{1}{|x|^{\mu}}*G(\e_{n} x, v^{2})\right)g(\e_{n} x, v^{2})v &\mbox{ in } \R^{N} \\
v\geq 0 \quad \mbox{ in } \R^{N}.
\end{array}
\right.
\end{equation}
Fix $\varphi\in C^{\infty}_{c}(\R^{N}, \R)$ such that $\varphi\geq 0$, and we take $\psi_{\delta, n}=\frac{u_{n}}{u_{\delta, n}}\varphi$ as test function in \eqref{Pe}, where $u_{\delta,n}=\sqrt{|u_{n}|^{2}+\delta^{2}}$ for $\delta>0$. We note that $\psi_{\delta, n}\in H^{s}_{\e_{n}}$ for all $\delta>0$ and $n\in \mathbb{N}$. 
Indeed, it is clear that
$$
\int_{\R^{N}} V(\e_{n} x) |\psi_{\delta,n}|^{2} dx\leq \int_{\supp(\varphi)} V(\e_{n} x)\varphi^{2} dx<\infty.
$$  
Now, we show that $[\psi_{\delta, n}]_{A_{\e}}$ is finite. Let us observe that
\begin{align*}
\psi_{\delta,n}(x)-\psi_{\delta,n}(y)e^{\imath A_{\e}(\frac{x+y}{2})\cdot (x-y)}&=\Bigl(\frac{u_{n}(x)}{u_{\delta,n}(x)}\Bigr)\varphi(x)-\Bigl(\frac{u_{n}(y)}{u_{\delta,n}(y)}\Bigr)\varphi(y)e^{\imath A_{\e}(\frac{x+y}{2})\cdot (x-y)}\\
&=\Bigl[\Bigl(\frac{u_{n}(x)}{u_{\delta,n}(x)}\Bigr)-\Bigl(\frac{u_{n}(y)}{u_{\delta,n}(x)}\Bigr)e^{\imath A_{\e}(\frac{x+y}{2})\cdot (x-y)}\Bigr]\varphi(x) \\
&+\Bigl[\varphi(x)-\varphi(y)\Bigr] \Bigl(\frac{u_{n}(y)}{u_{\delta,n}(x)}\Bigr) e^{\imath A_{\e}(\frac{x+y}{2})\cdot (x-y)} \\
&+\Bigl(\frac{u_{n}(y)}{u_{\delta,n}(x)}-\frac{u_{n}(y)}{u_{\delta,n}(y)}\Bigr)\varphi(y) e^{\imath A_{\e}(\frac{x+y}{2})\cdot (x-y)}.
\end{align*}
Then, by using $|z+w+k|^{2}\leq 4(|z|^{2}+|w|^{2}+|k|^{2})$ for all $z,w,k\in \C$, $|e^{\imath t}|=1$ for all $t\in \R$, $u_{\delta,n}\geq \delta$, $|\frac{u_{n}}{u_{\delta,n}}|\leq 1$, \eqref{UBu} and $|\sqrt{|z|^{2}+\delta^{2}}-\sqrt{|w|^{2}+\delta^{2}}|\leq ||z|-|w||$ for all $z, w\in \C$, we obtain that
\begin{align*}
&|\psi_{\delta,n}(x)-\psi_{\delta,n}(y)e^{\imath A_{\e}(\frac{x+y}{2})\cdot (x-y)}|^{2} \\
&\leq \frac{4}{\delta^{2}}|u_{n}(x)-u_{n}(y)e^{\imath A_{\e}(\frac{x+y}{2})\cdot (x-y)}|^{2}\|\varphi\|^{2}_{L^{\infty}(\R^{N})} +\frac{4}{\delta^{2}}|\varphi(x)-\varphi(y)|^{2} \|u_{n}\|^{2}_{L^{\infty}(\R^{N})} \\
&+\frac{4}{\delta^{4}} \|u_{n}\|^{2}_{L^{\infty}(\R^{N})} \|\varphi\|^{2}_{L^{\infty}(\R^{N})} |u_{\delta,n}(y)-u_{\delta,n}(x)|^{2} \\
&\leq \frac{4}{\delta^{2}}|u_{n}(x)-u_{n}(y)e^{\imath A_{\e}(\frac{x+y}{2})\cdot (x-y)}|^{2}\|\varphi\|^{2}_{L^{\infty}(\R^{N})} +\frac{4K^{2}}{\delta^{2}}|\varphi(x)-\varphi(y)|^{2} \\
&+\frac{4K^{2}}{\delta^{4}} \|\varphi\|^{2}_{L^{\infty}(\R^{N})} ||u_{n}(y)|-|u_{n}(x)||^{2}. 
\end{align*}
Since $u_{n}\in H^{s}_{\e_{n}}$, $|u_{n}|\in H^{s}(\R^{N}, \R)$ (by Lemma \ref{DI}) and $\varphi\in C^{\infty}_{c}(\R^{N}, \R)$, we conclude that $\psi_{\delta,n}\in H^{s}_{\e_{n}}$.
Then we get
\begin{align}\label{Kato1}
&\Re\Bigl[\iint_{\R^{2N}} \frac{(u_{n}(x)-u_{n}(y)e^{\imath A_{\e}(\frac{x+y}{2})\cdot (x-y)})}{|x-y|^{N+2s}} \times \nonumber \\
&\quad \times \Bigl(\frac{\overline{u_{n}(x)}}{u_{\delta,n}(x)}\varphi(x)\!-\!\frac{\overline{u_{n}(y)}}{u_{\delta,n}(y)}\varphi(y)e^{-\imath A_{\e}(\frac{x+y}{2})\cdot (x-y)}  \Bigr) dx dy\Bigr] 
+\int_{\R^{N}} V(\e_{n} x)\frac{|u_{n}|^{2}}{u_{\delta,n}}\varphi dx\nonumber\\
&=\int_{\R^{N}} \Bigl(\frac{1}{|x|^{\mu}}*G(\e_{n} x, |u_{n}|^{2})\Bigr) g(\e_{n} x, |u_{n}|^{2})\frac{|u_{n}|^{2}}{u_{\delta,n}}\varphi dx.
\end{align}
Now, we aim to pass to the limit as $\delta\rightarrow 0$ in \eqref{Kato1} to deduce that \eqref{Kato0} holds true.
Since $\Re(z)\leq |z|$ for all $z\in \C$ and  $|e^{\imath t}|=1$ for all $t\in \R$, we have
\begin{align}\label{alves1}
&\Re\Bigl[(u_{n}(x)-u_{n}(y)e^{\imath A_{\e}(\frac{x+y}{2})\cdot (x-y)}) \Bigl(\frac{\overline{u_{n}(x)}}{u_{\delta,n}(x)}\varphi(x)-\frac{\overline{u_{n}(y)}}{u_{\delta,n}(y)}\varphi(y)e^{-\imath A_{\e}(\frac{x+y}{2})\cdot (x-y)}  \Bigr)\Bigr] \nonumber\\
&=\Re\Bigl[\frac{|u_{n}(x)|^{2}}{u_{\delta,n}(x)}\varphi(x)+\frac{|u_{n}(y)|^{2}}{u_{\delta,n}(y)}\varphi(y)-\frac{u_{n}(x)\overline{u_{n}(y)}}{u_{\delta,n}(y)}\varphi(y)e^{-\imath A_{\e}(\frac{x+y}{2})\cdot (x-y)} \nonumber \\
&-\frac{u_{n}(y)\overline{u_{n}(x)}}{u_{\delta,n}(x)}\varphi(x)e^{\imath A_{\e}(\frac{x+y}{2})\cdot (x-y)}\Bigr] \nonumber \\
&\geq \Bigl[\frac{|u_{n}(x)|^{2}}{u_{\delta,n}(x)}\varphi(x)+\frac{|u_{n}(y)|^{2}}{u_{\delta,n}(y)}\varphi(y)-|u_{n}(x)|\frac{|u_{n}(y)|}{u_{\delta,n}(y)}\varphi(y)-|u_{n}(y)|\frac{|u_{n}(x)|}{u_{\delta,n}(x)}\varphi(x) \Bigr].
\end{align}
Let us note that
\begin{align}\label{alves2}
&\frac{|u_{n}(x)|^{2}}{u_{\delta,n}(x)}\varphi(x)+\frac{|u_{n}(y)|^{2}}{u_{\delta,n}(y)}\varphi(y)-|u_{n}(x)|\frac{|u_{n}(y)|}{u_{\delta,n}(y)}\varphi(y)-|u_{n}(y)|\frac{|u_{n}(x)|}{u_{\delta,n}(x)}\varphi(x) \nonumber\\
&=  \frac{|u_{n}(x)|}{u_{\delta,n}(x)}(|u_{n}(x)|-|u_{n}(y)|)\varphi(x)-\frac{|u_{n}(y)|}{u_{\delta,n}(y)}(|u_{n}(x)|-|u_{n}(y)|)\varphi(y) \nonumber\\
&=\Bigl[\frac{|u_{n}(x)|}{u_{\delta,n}(x)}(|u_{n}(x)|-|u_{n}(y)|)\varphi(x)-\frac{|u_{n}(x)|}{u_{\delta,n}(x)}(|u_{n}(x)|-|u_{n}(y)|)\varphi(y)\Bigr] \nonumber\\
&\quad +\Bigl(\frac{|u_{n}(x)|}{u_{\delta,n}(x)}-\frac{|u_{n}(y)|}{u_{\delta,n}(y)} \Bigr) (|u_{n}(x)|-|u_{n}(y)|)\varphi(y) \nonumber\\
&=\frac{|u_{n}(x)|}{u_{\delta,n}(x)}(|u_{n}(x)|-|u_{n}(y)|)(\varphi(x)-\varphi(y)) \nonumber \\
&\quad +\Bigl(\frac{|u_{n}(x)|}{u_{\delta,n}(x)}-\frac{|u_{n}(y)|}{u_{\delta,n}(y)} \Bigr) (|u_{n}(x)|-|u_{n}(y)|)\varphi(y) \nonumber\\
&\geq \frac{|u_{n}(x)|}{u_{\delta,n}(x)}(|u_{n}(x)|-|u_{n}(y)|)(\varphi(x)-\varphi(y)) 
\end{align}
where in the last inequality we used the fact that
$$
\left(\frac{|u_{n}(x)|}{u_{\delta,n}(x)}-\frac{|u_{n}(y)|}{u_{\delta,n}(y)} \right) (|u_{n}(x)|-|u_{n}(y)|)\varphi(y)\geq 0
$$
because
$$
h(t)=\frac{t}{\sqrt{t^{2}+\delta^{2}}} \mbox{ is increasing for } t\geq 0 \quad \mbox{ and } \quad \varphi\geq 0 \mbox{ in }\R^{N}.
$$
Since
\begin{align*}
&\frac{|\frac{|u_{n}(x)|}{u_{\delta,n}(x)}(|u_{n}(x)|-|u_{n}(y)|)(\varphi(x)-\varphi(y))|}{|x-y|^{N+2s}}\\
&\quad \leq \frac{||u_{n}(x)|-|u_{n}(y)||}{|x-y|^{\frac{N+2s}{2}}} \frac{|\varphi(x)-\varphi(y)|}{|x-y|^{\frac{N+2s}{2}}}\in L^{1}(\R^{2N}),
\end{align*}
and $\frac{|u_{n}(x)|}{u_{\delta,n}(x)}\rightarrow 1$ a.e. in $\R^{N}$ as $\delta\rightarrow 0$,
we can use \eqref{alves1}, \eqref{alves2} and the Dominated Convergence Theorem to deduce that
\begin{align}\label{Kato2}
&\limsup_{\delta\rightarrow 0} \Re\Bigl[\iint_{\R^{2N}} \frac{(u_{n}(x)-u_{n}(y)e^{\imath A_{\e}(\frac{x+y}{2})\cdot (x-y)})}{|x-y|^{N+2s}} \times \nonumber \\
&\quad \times \Bigl(\frac{\overline{u_{n}(x)}}{u_{\delta,n}(x)}\varphi(x)-\frac{\overline{u_{n}(y)}}{u_{\delta,n}(y)}\varphi(y)e^{-\imath A_{\e}(\frac{x+y}{2})\cdot (x-y)}  \Bigr) dx dy\Bigr] \nonumber\\
&\geq \limsup_{\delta\rightarrow 0} \iint_{\R^{2N}} \frac{|u_{n}(x)|}{u_{\delta,n}(x)}(|u_{n}(x)|-|u_{n}(y)|)(\varphi(x)-\varphi(y)) \frac{dx dy}{|x-y|^{N+2s}} \nonumber\\
&=\iint_{\R^{2N}} \frac{(|u_{n}(x)|-|u_{n}(y)|)(\varphi(x)-\varphi(y))}{|x-y|^{N+2s}} dx dy.
\end{align}
On the other hand, from the Dominated Convergence Theorem (we note that $\frac{|u_{n}|^{2}}{u_{\delta, n}}\leq |u_{n}|$, $\varphi\in C^{\infty}_{c}(\R^{N}, \R)$ and $\tilde{K}_{\e}(u_{n})$ is bounded in view of Lemma \ref{lemK}) we can infer that
\begin{equation}\label{Kato3}
\lim_{\delta\rightarrow 0} \int_{\R^{N}} V(\e_{n} x)\frac{|u_{n}|^{2}}{u_{\delta,n}}\varphi dx=\int_{\R^{N}} V(\e_{n} x)|u_{n}|\varphi dx
\end{equation}
and
\begin{align}\label{Kato4}
&\lim_{\delta\rightarrow 0}  \int_{\R^{N}}  \left(\frac{1}{|x|^{\mu}}*G(\e_{n} x, |u_{n}|^{2})\right) g(\e_{n} x, |u_{n}|^{2})\frac{|u_{n}|^{2}}{u_{\delta,n}}\varphi dx\nonumber \\
&=\int_{\R^{N}}  \left(\frac{1}{|x|^{\mu}}*G(\e_{n} x, |u_{n}|^{2})\right) g(\e_{n} x, |u_{n}|^{2}) |u_{n}|\varphi dx.
\end{align}
Taking into account \eqref{Kato1}, \eqref{Kato2}, \eqref{Kato3} and \eqref{Kato4} we can see that
\begin{align*}
&\iint_{\R^{2N}} \frac{(|u_{n}(x)|-|u_{n}(y)|)(\varphi(x)-\varphi(y))}{|x-y|^{N+2s}} dx dy+\int_{\R^{N}} V(\e_{n} x)|u_{n}|\varphi dx\\
&\leq \int_{\R^{N}}  \left(\frac{1}{|x|^{\mu}}*G(\e_{n} x, |u_{n}|^{2})\right) g(\e_{n} x, |u_{n}|^{2}) |u_{n}|\varphi dx
\end{align*}
for any $\varphi\in C^{\infty}_{c}(\R^{N}, \R)$ such that $\varphi\geq 0$. Then $|u_{n}|$ is a weak subsolution to \eqref{Kato0}.
By using $(V_{1})$, $u_{n}\in \B$ for all $n$ big enough, and Lemma \ref{lemK}, it is clear that $v_{n}=|u_{n}|(\cdot+\tilde{y}_{n})$ solves 
\begin{equation}\label{Pkat}
(-\Delta)^{s} v_{n} + V_{0}v_{n}\leq C_{0} g(\e_{n} x+\e_{n}\tilde{y}_{n}, v_{n}^{2})v_{n} \mbox{ in } \R^{N}. 
\end{equation}
Let us denote by $z_{n}\in H^{s}(\R^{N}, \R)$ the unique solution to
\begin{equation}\label{US}
(-\Delta)^{s} z_{n} + V_{0}z_{n}= C_{0}g_{n} \mbox{ in } \R^{N},
\end{equation}
where
$$
g_{n}:=g(\e_{n} x+\e_{n}\tilde{y}_{n}, v_{n}^{2})v_{n}\in L^{r}(\R^{N}, \R) \quad \forall r\in [2, \infty].
$$
Since \eqref{UBu} yields $\|v_{n}\|_{L^{\infty}(\R^{N})}\leq C$ for all $n\in \mathbb{N}$, by interpolation we know that $v_{n}\rightarrow v$ strongly converges in $L^{r}(\R^{N}, \R)$ for all $r\in (2, \infty)$, for some $v\in L^{r}(\R^{N}, \R)$, and by the growth assumptions on $f$, we can see that also $g_{n}\rightarrow  f(v^{2})v$ in $L^{r}(\R^{N}, \R)$ and $\|g_{n}\|_{L^{\infty}(\R^{N})}\leq C$ for all $n\in \mathbb{N}$.
Since $z_{n}=\mathcal{K}*(C_{0}g_{n})$, where $\mathcal{K}$ is the Bessel kernel (see \cite{FQT}), we can argue as in \cite{AM} to infer that $|z_{n}(x)|\rightarrow 0$ as $|x|\rightarrow \infty$ uniformly with respect to $n\in \mathbb{N}$.
On the other hand, $v_{n}$ satisfies \eqref{Pkat} and $z_{n}$ solves \eqref{US} so a simple comparison argument shows that $0\leq v_{n}\leq z_{n}$ a.e. in $\R^{N}$ and for all $n\in \mathbb{N}$. This means that $v_{n}(x)\rightarrow 0$ as $|x|\rightarrow \infty$ uniformly in $n\in \mathbb{N}$.
\end{proof}

\noindent
At this point we have all ingredients to give the proof of Theorem \ref{thm1}.
\begin{proof}
By using Lemma \ref{prop3.3} we can find a sequence $(\tilde{y}_{n})\subset \R^{N}$ such that $\e_{n}\tilde{y}_{n}\rightarrow y_{0}$ for some $y_{0} \in \Lambda$ such that $V(y_{0})=V_{0}$. 
Then there exists $r>0$ such that, for some subsequence still denoted by itself, we have $B_{r}(\tilde{y}_{n})\subset \Lambda$ for all $n\in \mathbb{N}$.
Hence $B_{\frac{r}{\e_{n}}}(\tilde{y}_{n})\subset \Lambda_{\e_{n}}$ for all $n\in \mathbb{N}$, which gives 
$$
\R^{N}\setminus \Lambda_{\e_{n}}\subset \R^{N} \setminus B_{\frac{r}{\e_{n}}}(\tilde{y}_{n}) \mbox{ for any } n\in \mathbb{N}.
$$ 
In view of Lemma \ref{moser}, we know that there exists $R>0$ such that 
$$
v_{n}(x)<a \mbox{ for } |x|\geq R \mbox{ and } n\in \mathbb{N},
$$ 
where $v_{n}(x)=|u_{\e_{n}}|(x+ \tilde{y}_{n})$. 
Then $|u_{\e_{n}}(x)|<a$ for any $x\in \R^{N}\setminus B_{R}(\tilde{y}_{n})$ and $n\in \mathbb{N}$. Moreover, there exists $\nu \in \mathbb{N}$ such that for any $n\geq \nu$ and $r/\e_{n}>R$ it holds 
$$
\R^{N}\setminus \Lambda_{\e_{n}}\subset \R^{N} \setminus B_{\frac{r}{\e_{n}}}(\tilde{y}_{n})\subset \R^{N}\setminus B_{R}(\tilde{y}_{n}),
$$ 
which gives $|u_{\e_{n}}(x)|<a$ for any $x\in \R^{N}\setminus \Lambda_{\e_{n}}$ and $n\geq \nu$.  \\
Therefore, there exists $\e_{0}>0$ such that problem \eqref{R} admits a nontrivial solution $u_{\e}$ for all $\e\in (0, \e_{0})$. Then $\hat{u}_{\e}(x)=u_{\e}(x/\e)$ is a solution to (\ref{P}). 
Finally, we study the behavior of the maximum points of  $|u_{\e_{n}}|$. In view of $(g_1)$, there exists $\gamma\in (0,a)$ such that
\begin{align}\label{4.18HZ}
g(\e x, t^{2})t^{2}\leq \frac{V_{0}}{\ell_{0}}t^{2}, \mbox{ for all } x\in \R^{N}, |t|\leq \gamma.
\end{align}
Using a similar discussion above, we can take $R>0$ such that
\begin{align}\label{4.19HZ}
\|u_{\e_{n}}\|_{L^{\infty}(B^{c}_{R}(\tilde{y}_{n}))}<\gamma.
\end{align}
Up to a subsequence, we may also assume that
\begin{align}\label{4.20HZ}
\|u_{\e_{n}}\|_{L^{\infty}(B_{R}(\tilde{y}_{n}))}\geq \gamma.
\end{align}
Otherwise, if \eqref{4.20HZ} does not hold, then $\|u_{\e_{n}}\|_{L^{\infty}(\R^{N})}< \gamma$ and by using $J_{\e_{n}}'(u_{\e_{n}})=0$, \eqref{4.18HZ}, Lemma \ref{DI} and 
$$
\left\|\frac{1}{|x|^{\mu}}*G(\e x, |u_{n}|^{2})\right\|_{L^{\infty}(\R^{N})}<C_{0},
$$
we have
\begin{align*}
[|u_{\e_{n}}|]^{2}+\int_{\R^{N}}V_{0}|u_{\e_{n}}|^{2}dx\leq \|u_{\e_{n}}\|^{2}_{\e_{n}}&=\int_{\R^{N}} \tilde{K}_{\e}(u_{\e_{n}}) g_{\e_{n}}(x, |u_{\e_{n}}|^{2})|u_{\e_{n}}|^{2}\,dx\\
&\leq \frac{C_{0}V_{0}}{\ell_{0}}\int_{\R^{N}}|u_{\e_{n}}|^{2}\, dx
\end{align*}
and being $\frac{C_{0}}{\ell_{0}}<\frac{1}{2}$ we deduce that $\||u_{\e_{n}}|\|_{H^{s}(\R^{N})}=0$ which is impossible. 
From \eqref{4.19HZ} and \eqref{4.20HZ}, it follows  that the maximum points $p_{n}$ of $|u_{\e_{n}}|$ belong to $B_{R}(\tilde{y}_{n})$, that is $p_{n}=\tilde{y}_{n}+q_{n}$ for some $q_{n}\in B_{R}$.  Since $\hat{u}_{n}(x)=u_{\e_{n}}(x/\e_{n})$ is a solution to \eqref{P}, we can see that the maximum point $\eta_{\e_{n}}$ of $|\hat{u}_{n}|$ is given by $\eta_{\e_{n}}=\e_{n}\tilde{y}_{n}+\e_{n}q_{n}$. Taking into account $q_{n}\in B_{R}$, $\e_{n}\tilde{y}_{n}\rightarrow y_{0}$ and $V(y_{0})=V_{0}$ and  the continuity of $V$, we can infer that
$$
\lim_{n\rightarrow \infty} V(\eta_{\e_{n}})=V_{0}.
$$
Finally, we give a  decay estimate for $|\hat{u}_{n}|$. We follow some arguments used  in \cite{A6}.\\
Invoking Lemma $4.3$ in \cite{FQT}, we can find a function $w$ such that 
\begin{align}\label{HZ1}
0<w(x)\leq \frac{C}{1+|x|^{N+2s}},
\end{align}
and
\begin{align}\label{HZ2}
(-\Delta)^{s} w+\frac{V_{0}}{2}w\geq 0 \mbox{ in } \R^{N}\setminus B_{R_{1}} 
\end{align}
for some suitable $R_{1}>0$. Using Lemma \ref{moser}, we know that $v_{n}(x)\rightarrow 0$ as $|x|\rightarrow \infty$ uniformly in $n\in \mathbb{N}$, so there exists $R_{2}>0$ such that
\begin{equation}\label{hzero}
h_{n}=C_{0}g(\e_{n}x+\e_{n}\tilde{y}_{n}, v_{n}^{2})v_{n}\leq \frac{C_{0}V_{0}}{\ell_{0}}v_{n}\leq \frac{V_{0}}{2}v_{n}  \mbox{ in } B_{R_{2}}^{c}.
\end{equation}
Let us denote by $w_{n}$ the unique solution to 
$$
(-\Delta)^{s}w_{n}+V_{0}w_{n}=h_{n} \mbox{ in } \R^{N}.
$$
Then $w_{n}(x)\rightarrow 0$ as $|x|\rightarrow \infty$ uniformly in $n\in \mathbb{N}$, and by comparison $0\leq v_{n}\leq w_{n}$ in $\R^{N}$. Moreover, in view of \eqref{hzero}, it holds
\begin{align*}
(-\Delta)^{s}w_{n}+\frac{V_{0}}{2}w_{n}=h_{n}-\frac{V_{0}}{2}w_{n}\leq 0 \mbox{ in } B_{R_{2}}^{c}.
\end{align*}
Take $R_{3}=\max\{R_{1}, R_{2}\}$ and we define
\begin{align}\label{HZ4}
a=\inf_{B_{R_{3}}} w>0 \mbox{ and } \tilde{w}_{n}=(b+1)w-a w_{n}.
\end{align}
where $b=\sup_{n\in \mathbb{N}} \|w_{n}\|_{L^{\infty}(\R^{N})}<\infty$. 
We aim to prove that 
\begin{equation}\label{HZ5}
\tilde{w}_{n}\geq 0 \mbox{ in } \R^{N}.
\end{equation}
Let us note that
\begin{align}
&\lim_{|x|\rightarrow \infty} \tilde{w}_{n}(x)=0 \mbox{ uniformly in } n\in \mathbb{N},  \label{HZ0N} \\
&\tilde{w}_{n}\geq ba+w-ba>0 \mbox{ in } B_{R_{3}} \label{HZ0},\\
&(-\Delta)^{s} \tilde{w}_{n}+\frac{V_{0}}{2}\tilde{w}_{n}\geq 0 \mbox{ in } \R^{N}\setminus B_{R_{3}} \label{HZ00}.
\end{align}
We argue by contradiction, and we assume that there exists a sequence $(\bar{x}_{j, n})\subset \R^{N}$ such that 
\begin{align}\label{HZ6}
\inf_{x\in \R^{N}} \tilde{w}_{n}(x)=\lim_{j\rightarrow \infty} \tilde{w}_{n}(\bar{x}_{j, n})<0. 
\end{align}
From (\ref{HZ0N}) it follows that $(\bar{x}_{j, n})$ is bounded, and, up to subsequence, we may assume that there exists $\bar{x}_{n}\in \R^{N}$ such that $\bar{x}_{j, n}\rightarrow \bar{x}_{n}$ as $j\rightarrow \infty$. 
In view of (\ref{HZ6}) we can see that
\begin{align}\label{HZ7}
\inf_{x\in \R^{N}} \tilde{w}_{n}(x)= \tilde{w}_{n}(\bar{x}_{n})<0.
\end{align}
By using the minimality of $\bar{x}_{n}$ and the integral representation formula for the fractional Laplacian \cite{DPV}, we can see that 
\begin{align}\label{HZ8}
(-\Delta)^{s}\tilde{w}_{n}(\bar{x}_{n})=\frac{C(N, s)}{2} \int_{\R^{N}} \frac{2\tilde{w}_{n}(\bar{x}_{n})-\tilde{w}_{n}(\bar{x}_{n}+\xi)-\tilde{w}_{n}(\bar{x}_{n}-\xi)}{|\xi|^{N+2s}} d\xi\leq 0.
\end{align}
Putting together (\ref{HZ0}) and (\ref{HZ6}), we have $\bar{x}_{n}\in \R^{N}\setminus B_{R_{3}}$.
This fact combined with (\ref{HZ7}) and (\ref{HZ8}) yields 
$$
(-\Delta)^{s} \tilde{w}_{n}(\bar{x}_{n})+\frac{V_{0}}{2}\tilde{w}_{n}(\bar{x}_{n})<0,
$$
which gives a contradiction in view of (\ref{HZ00}).
As a consequence (\ref{HZ5}) holds true, and by using (\ref{HZ1}) and $v_{n}\leq w_{n}$ we can deduce that
\begin{align*}
0\leq v_{n}(x)\leq w_{n}(x)\leq \frac{(b+1)}{a}w(x)\leq \frac{\tilde{C}}{1+|x|^{N+2s}} \mbox{ for all } n\in \mathbb{N}, x\in \R^{N},
\end{align*}
for some constant $\tilde{C}>0$. Recalling the definition of $v_{n}$, we can obtain that  
\begin{align*}
|\hat{u}_{n}(x)|&=\left|u_{\e_{n}}\left(\frac{x}{\e_{n}}\right)\right|=v_{n}\left(\frac{x}{\e_{n}}-\tilde{y}_{n}\right) \\
&\leq \frac{\tilde{C}}{1+|\frac{x}{\e_{n}}-\tilde{y}_{\e_{n}}|^{N+2s}} \\
&=\frac{\tilde{C} \e_{n}^{N+2s}}{\e_{n}^{N+2s}+|x- \e_{n} \tilde{y}_{\e_{n}}|^{N+2s}} \\
&\leq \frac{\tilde{C} \e_{n}^{N+2s}}{\e_{n}^{N+2s}+|x-\eta_{\e_{n}}|^{N+2s}}.
\end{align*}

\end{proof}

\noindent {\bf Acknowledgements.} 
The author thanks Claudianor O. Alves and Hoai-Minh Nguyen for delightful and pleasant discussions about the results of this work.

\end{document}